\documentclass[11pt]{amsart}
\usepackage{amsmath}
\usepackage{amssymb, verbatim}
\usepackage{pstricks,pst-node,pst-plot}
\usepackage{comment}

\title[$(1,1)$ forms with specified Lagrangian phase]{$(1,1)$ forms with specified Lagrangian phase: A priori estimates and algebraic obstructions.}
\author[T.C. Collins]{Tristan C. Collins}
  \email{tcollins@math.harvard.edu}
  \thanks{T.C.C is supported in part by NSF grant DMS-1506652.  A.J. is supported in part by NSF grant DMS-1204155.  S.-T. Y. is supported in part by NSF grant DMS-1308244}
\author[A. Jacob]{Adam Jacob}

  \email{ajacob@math.harvard.edu}
  \author[S.-T Yau]{Shing-Tung Yau}
  \email{yau@math.harvard.edu}
  \address{Department of Mathematics, Harvard University, 1 Oxford Street, Cambridge, MA 02138}

\theoremstyle{plain}
\newtheorem{thm}{Theorem}[section]
\newtheorem{prop}[thm]{Proposition}
\newtheorem{defn}[thm]{Definition}
\newtheorem{lem}[thm]{Lemma}
\newtheorem{cor}[thm]{Corollary}
\newtheorem{conj}[thm]{Conjecture}

\theoremstyle{definition}

\newtheorem{rk}[thm]{Remark}

\numberwithin{equation}{section}

\newcommand{\Tr}{\textrm{Tr}}

\newcommand{\del}{\partial}
\newcommand{\de}{\partial}
\newcommand{\la}{\lambda}
\newcommand{\dbar}{\overline{\del}}
\newcommand{\ddb}{\sqrt{-1}\del\dbar}

\newcommand{\ti}[1]{\tilde{#1}}

\renewcommand{\leq}{\leqslant}
\renewcommand{\geq}{\geqslant}

\renewcommand{\epsilon}{\varepsilon}
\renewcommand{\phi}{\varphi}

\begin{document}

\begin{abstract}
Let $(X,\alpha)$ be a K\"ahler manifold of dimension $n$, and let $[\omega] \in H^{1,1}(X,\mathbb{R})$.  We study the problem of specifying the Lagrangian phase of $\omega$ with respect to $\alpha$, which is described by the nonlinear elliptic equation
\[
\sum_{i=1}^{n} \arctan(\la_i)= h(x)
\]
where $\la_i$ are the eigenvalues of $\omega$ with respect to $\alpha$.  When $h(x)$ is a topological constant, this equation corresponds to the deformed Hermitian-Yang-Mills (dHYM) equation, and is related by Mirror Symmetry to the existence of special Lagrangian submanifolds of the mirror.  We introduce a notion of subsolution for this equation, and prove a priori $C^{2,\beta}$ estimates when $|h|>(n-2)\frac{\pi}{2}$ and a subsolution exists.  Using the method of continuity we show that the dHYM equation admits a smooth solution in the supercritical phase case, whenever a subsolution exists.  Finally, we discover some stability-type cohomological obstructions to the existence of solutions to the dHYM equation and we conjecture that when these obstructions vanish the dHYM equation admits a solution.  We confirm this conjecture for complex surfaces.

\end{abstract}
\maketitle

\section{Introduction}

Let $(X,\alpha)$ be a connected, compact K\"ahler manifold of complex dimension $n$, and let $\Omega\in H^{1,1}(X,\mathbb{R})$ be a cohomology class.  Motivated by Mirror Symmetry, the second and third authors introduced the following problem in \cite{JY};  does there exist a smooth, closed $(1,1)$ form $\omega$, such that $[\omega]= \Omega$, and
\begin{equation}\label{eq: form type}
{\rm Im}(\alpha+ \sqrt{-1}\omega)^n = \tan( \hat{\theta}) {\rm Re}(\alpha +\sqrt{-1}\omega)^n,
\end{equation}
where $\hat{\theta}$ is a topological constant determined by $[\alpha],\Omega$ ?  When $\Omega\in H^{1,1}(X,\mathbb{Z})$, this equation is known as the {\em deformed Hermitian-Yang-Mills equation}, and plays an important role in Mirror Symmetry.  Written in terms of the eigenvalues of the relative endomorphism $\Lambda^{j}_{k} = \alpha^{j\bar{\ell}}\omega_{k\bar{\ell}}$, equation~\eqref{eq: form type} can be written as \cite{JY}
\begin{equation}\label{eq: SLag type}
\Theta_{\alpha}(\omega):= \sum_{i=1}^{n} \arctan(\la_i) = \hat{\Theta}.
\end{equation}
Equation~\eqref{eq: SLag type} is the natural generalization to compact K\"ahler manifolds of the special Lagrangian equation with potential introduced by Harvey-Lawson \cite{HL} and since studied extensively; see, for instance, \cite{CNS, Sm, SmW, WY, WY1, MTW,Y1,Y} and the references therein.  The third author, with Leung and Zaslow \cite{LYZ}, showed that when $\Omega =c_1(L)$ for a holomorphic line bundle $L \rightarrow X$, and $X$ is a torus fibration over a torus, solutions of equation~\eqref{eq: SLag type} are related via the Fourier-Mukai transform to special Lagrangian sections of the dual torus fibration. In their paper \cite{JY}, the second and third authors initiated the study of~\eqref{eq: SLag type} over a compact K\"ahler manifold, and using a parabolic flow they proved the existence of solutions when $(X,\alpha)$ has positive bisectional curvature, and the initial data is sufficiently positive.  

The starting point for this work is the following simple observation;  suppose $\Omega = c_{1}(L)$ is K\"ahler, and look for hermitian metrics $h$ on $L$ whose curvature form satisfies~\eqref{eq: form type}.  We may also look for metrics $h$ on $L$, so that the curvature form of $h^{k}$ on $L^{\otimes k}$ satisfies equation~\eqref{eq: form type} with $c_{1}(L)$ replaced by $kc_{1}(L)$.  These two equations are {\em different}.  It is therefore natural to ask for the limiting equation when $k\rightarrow \infty$.  Multiplying both sides of~\eqref{eq: form type} by $k^{-n}$, and taking a limit as $k\rightarrow \infty$ one easily obtains that the limiting equation is
\begin{equation}\label{eq: J-eqn}
c\omega^{n} = n\omega^{n-1}\wedge \alpha,
\end{equation}
for $\omega \in c_{1}(L)$ with $c$ a topological constant.  Equation~\eqref{eq: J-eqn} is precisely the $J$-equation, discovered independently by Donaldson \cite{Don} and Chen \cite{Ch1, Ch2}.  Let us briefly recall some of the important analytic and algebraic facts about the $J$-equation to serve as motivation for our work.  Analytically, the solvability of the $J$-equation on general compact K\"ahler manifolds is well understood thanks to work of Song-Weinkove \cite{SW}, building on previous work of Weinkove \cite{W1,W2}.  Song-Weinkove show that the existence of a solution to the $J$-equation is equivalent to the existence of a K\"ahler metric $\chi\in[\omega]$ with
\begin{equation}\label{eq: J-eqn subsol}
c\chi^{n-1} - (n-1)\chi^{n-2}\wedge \alpha >0
\end{equation}
in the sense of $(n-1,n-1)$ forms.  Very recently, Sz\'ekelyhidi \cite{Szek} introduced a notion of subsolutions for a very general class of Hessian type equations on Hermitian manifolds, which encompasses~\eqref{eq: J-eqn subsol}, and showed that the existence of a subsolution implies a priori estimates to all orders.  

The primary goal in this work is to begin building an analytic and algebraic framework for studying the existence problem for solutions of equation~\eqref{eq: SLag type}.  As a first step, we study the specified Lagrangian phase equation;
\begin{equation}\label{eq: Spec Lag Ang}
\Theta_{\alpha}(\omega) :=\sum_{i}\arctan(\la_i) = h(x).
\end{equation}

Our first theorem is that, under the assumption of a subsolution, solutions of the specified Lagrangian phase equation with critical phase admit a priori estimates to all orders.

\begin{thm}\label{thm: est thm}
Fix $\omega_0\in\Omega$. Let $u:X \rightarrow \mathbb{R}$ be a smooth function such that $\sup_{X}u=0$ and $\Theta_{\alpha}(\omega_0+\ddb u) = h(x)$, where $h:X \rightarrow [(n-2)\frac{\pi}{2}+\epsilon_0, n\frac{\pi}{2})$.  Suppose there exists a $\mathcal{C}$-subsolution $\underline{u}:X\rightarrow \mathbb{R}$ in the sense of Definition~\ref{def: C-sub} (see also Lemma~\ref{lem: sub sol def}). Then for every $\beta \in (0,1)$ we have an estimate
\[
\|u\|_{C^{2,\beta}}  \leq C(X,\alpha, \beta, h, \epsilon_0, \omega_0, \underline{u}).
\]
\end{thm} 

Our notion of a subsolution is certainly necessary for the existence of a solution to~\eqref{eq: Spec Lag Ang}, and furthermore agrees with the notion of a $\mathcal{C}$-subsolution recently introduced by Sz\'ekelyhidi \cite{Szek}.  The Lagrangian phase equation~\eqref{eq: Spec Lag Ang} fails several of the structural conditions imposed in \cite{Szek}-- most seriously, in general, the operator we study fails to be concave.  The main difficulty in the proof of Theorem~\ref{thm: est thm} is the $C^2$ estimate which is rather delicate owing to the lack of concavity.  In the real case, a priori second order estimates for graphical solutions of the special Lagrangian equation with constant and critical phase  are proved by Wang-Yuan \cite{WY}.  By contrast, the complex setting studied here introduces several new negative terms into the estimate, which together with the non-constant phase, further complicate the analysis.

We apply these a priori estimates together with the method of continuity to prove an existence theorem for the deformed Hermitian-Yang-Mills equation

\begin{thm}\label{thm: existence thm}
Fix $\omega_0\in\Omega$ and suppose the topological constant $\hat{\Theta}$ satisfies the critical phase condition
\[
\hat{\Theta} >(n-2)\frac{\pi}{2}.
\]
Furthermore, suppose that there exists $\chi:=\omega_0+\sqrt{-1}\partial\bar\partial\underline{u} \in \Omega$ defining a subsolution  in the sense of Definition~\ref{def: C-sub} (see also Lemma~\ref{lem: sub sol def}).   Assume that either
\begin{itemize}
\item $\Theta_{\alpha}(\chi) >(n-2)\frac{\pi}{2}$, or
\item $\hat{\Theta} \geq ((n-2)+\frac{2}{n}) \frac{\pi}{2}$.
\end{itemize}
Then there exists a unique smooth $(1,1)$ form $\omega $ with $[\omega]=\Omega$ solving the deformed Hermitian-Yang-Mills equation
\[
\Theta_{\alpha}(\omega) = \hat{\Theta}.
\]
\end{thm}

We remark that in the statement of Theorem~\ref{thm: existence thm}, the appearance of two conditions is rather artificial.  In reality, we require only the first condition.  We have only added the second condition to emphasize that if $\chi \in \Omega$ is a subsolution to the deformed Hermitian-Yang-Mills equation with $\hat{\Theta} \geq [(n-2)+\frac{2}{n}] \frac{\pi}{2}$, then the first condition is automatically satisfied.  This result removes the hypercritical phase, non-negative sectional curvature, and large initial angle assumptions from \cite{JY}.  We expect that this result can be improved when the angle $\hat{\Theta} \in ((n-2)\frac{\pi}{2}, ((n-2)+\frac{2}{n}) \frac{\pi}{2})$, to remove the assumption that the subsolution has critical phase.  This expectation has been verified in dimension 2 \cite[Theorem 1.2]{JY}, and in dimension 3 \cite{P}, where it follows from work of Fang-Lai-Ma \cite{FLM}.

In the case of a domain in $\mathbb{C}^{n}$, we expect the natural extension of the subsolution condition considered here to be equivalent to the solvability of the boundary value problem, in analogy with the work of Guan-Li \cite{GL} on the inverse Hessian type equations.  In the real setting, the Dirichlet problem posed by Harvey-Lawson \cite{HL} was solved by Caffarelli-Niremberg-Spruck \cite{CNS} under some assumptions on the convexity of the boundary.  It is interesting to note the similarities between these convexity conditions and the subsolution condition in Lemma~\ref{lem: sub sol def}.  

Finally, we show that the existence of a subsolution imposes some cohomological restrictions on $X$.  In particular, we prove the following simple

\begin{prop}\label{prop: intro stable nec}
For every subvariety $V \subset X$, define 
\begin{equation}
\Theta_{V} := {\rm Arg} \int_{V} (\alpha+\sqrt{-1}\omega)^{\dim V}.
\end{equation}
If there exists a solution to the deformed Hermitian-Yang-Mills equation \eqref{eq: SLag type}, then for every proper subvariety $V\subset X$ we have
\[
\Theta_{V} > \Theta_{X} - (n-\dim V)\frac{\pi}{2}.
\]
\end{prop}

This condition is a close analog of the stability condition for the $J$-equation recently discovered by Lejmi-Sz\'ekelyhidi \cite{LS}, and we expect the obstruction in Proposition~\ref{prop: intro stable nec} to arise from a suitable adaptation of the K-stability framework, a problem we plan to address in future work.  In light of \cite[Conjecture 1]{LS}, and recent evidence for this conjecture by the first author and Sz\'ekelyhidi \cite{CS} and Lejmi-Sz\'ekelyhidi \cite{LS}, it does not seem irresponsible to pose

\begin{conj}\label{conj: Lag stab}
A solution of the deformed Hermitian-Yang-Mills equation~\eqref{eq: SLag type} exists if and only if for every proper subvariety $V\subset X$ we have
\[
\Theta_{V} > \Theta_{X} - (n-\dim V)\frac{\pi}{2},
\]
in the notation of Proposition~\ref{prop: intro stable nec}.
\end{conj}
In Proposition~\ref{prop:  conj dim 2} we show that this conjecture holds in dimension 2.  Furthermore, we briefly discuss how the stability condition can be interpreted in terms of a central charge.  In future work we hope to understand how Conjecture~\ref{conj: Lag stab} fits into the Mirror Symmetry setting for special Lagrangians and the conjectural picture put forth by Thomas and the third author \cite{TY}, and Thomas \cite{T1,T}.  Finally, we remark that there has recently been considerable interest in the analogy between the problem of finding special Lagrangians in a Calabi-Yau, and that of finding K\"ahler-Einstein or constant scalar curvature K\"ahler metrics as outlined by Solomon \cite{Sol1, Sol2}, and studied in recent work of Rubinstein-Solomon \cite{RS}.

The layout of this paper is as follows; in Section 2 we briefly discuss some background material, mostly taken from earlier work of the second and third authors \cite{JY}.  In Section 3 we discuss the notion of a $\mathcal{C}$-subsolution, and extract the results from \cite{Szek} which we will need.  In Section 4 we prove an a priori $C^2$ estimate in terms of the gradient for solutions of the specified Lagrangian phase equation~\eqref{eq: Spec Lag Ang}.  This is the most difficult step in the proof of Theorem~\ref{thm: est thm}.  In Section 5 we use a blow-up argument to prove an a priori gradient bound for solutions of~\eqref{eq: Spec Lag Ang}, which implies a uniform $C^2$ estimates.  In Section 6 we discuss the $C^{2,\beta}$ estimates, which follow from the usual Evans-Krylov estimate by a blow-up argument and a reduction to the real case.  In Section 7 we take up the method of continuity and prove Theorem~\ref{thm: existence thm}.  This actually turns out to be slightly involved, as the natural method of continuity does not obviously preserve the critical phase condition, nor the existence of a subsolution.  Instead we adapt a trick of Sun \cite{WS}, and use a double method of conintuity.  The first continuity path is used to find a suitable starting point for the second method of continuity, whose ending point is the solution of the deformed Hermitian-Yang-Mills equation.  In Section 8 we further discuss the implications of the existence of a subsolution for the deformed Hermitian-Yang-Mills equation, and deduce some algebraic obstructions to the existence of $(1,1)$ forms with constant Lagrangian phase.  We prove Proposition~\ref{prop: intro stable nec}, and give some evidence for Conjecture~\ref{conj: Lag stab}.
\\

{\bf Acknowledgements:} The first author is grateful to Gabor Sz\'ekelyhidi for several helpful conversations.  The authors would like to thank Valentino Tosatti and Ben Weinkove for helpful comments.

\section{Background and Notation}

\label{background}

Let us briefly discuss our setup.  Fix a compact K\"ahler manifold $X$ with K\"ahler form $\alpha$, and assume the normalization $\int_{X}\alpha^{n}=n!$. Throughout this paper, unless otherwise noted, the covariant derivative $\nabla$ and all norms are computed with respect to $\alpha$. 

Fix a cohomology class $\Omega \in H^{1,1}(X,\mathbb{R})$. The deformed Hermitian-Yang-Mills equation seeks a $(1,1)$ form $\omega \in \Omega$ wth the property that the map
\[
X \ni x \longmapsto\frac{ (\alpha+\sqrt{-1}\omega)^{n}}{\alpha^{n}} \in \mathbb{C}
\]
has constant argument.  Here we view ${\rm Arg}$ as a map from $\mathbb{C} \rightarrow \mathbb{R}$ (rather than $S^1$) by specifying that the argument of the function $1 = \frac{\alpha^{n}}{\alpha^{n}}$ is zero.  If a solution of this equation exists, then we necessarily have
\[
{\rm Arg} \frac{(\alpha+\sqrt{-1}\omega)^{n}}{\alpha^{n}} = {\rm Arg}\int_{X}\frac{ (\alpha+\sqrt{-1}\omega)^{n}}{n!} =: \hat{\Theta},
\]
where $\hat\Theta$ is a topological constant.  As shown in \cite{JY}, this problem is equivalent to both equations~\eqref{eq: form type} and~\eqref{eq: SLag type}.  We will primarily deal with the later representation.  As we discussed in the introduction, it is also necessary to consider the specified Lagrangian phase equation for non-constant phase
\[
\Theta_{\alpha}(\omega) := \sum_{i=1}^{n} \arctan(\la_i) = h(x),
\]
where again $\la_i$ are the eigenvalues of $\alpha^{-1}\omega$.  

It is useful to introduce another Hermitian metric on $T^{1,0}(X)$, defined by the formula $\eta_{j\bar{k}} = \alpha_{j\bar{k}} + \omega_{j\bar{\ell}}\alpha^{p\bar{\ell}} \omega_{p\bar{k}}.$ Note this metric is never K\"ahler. With this definition, following \cite{JY} one can compute the variation of $\Theta_{\alpha}$ as
\begin{equation}\label{eq: derivtheta}
\delta\Theta_{\alpha}= \eta^{j\bar{k}}\alpha_{\ell\bar k} \delta(\alpha^{\ell\bar m}\omega_{j\bar m}).
\end{equation}
This computation has two important consequences. First, using the covariant derivative $\nabla$ with respect to $\alpha$, one sees that $d\Theta=\eta^{j\bar{k}}\nabla\omega_{j\bar k}$. Furthermore, since we consider variations of $\omega$ which fix $\alpha$, the linearization of the operator $\Theta_{\alpha}(\omega)$ is given by
\begin{equation}\label{eq: lin op}
\Delta_{\eta} = \eta^{j\bar{k}} \del_j\del_{\bar{k}}.
\end{equation}
It is easy to check that this operator becomes uniformly elliptic as soon as $|\omega|$ is bounded.   At a point $x_0$ in coordinates where $\alpha(x_0)$ is the identity and $\omega(x_0)$ is diagonal with entries $\la_i$, then the metric $\eta_{j\bar k}$ is diagonal with entries
\[
\eta_{i\bar{i}}(x_0) = (1+\la_i^2)\delta_{i\bar{i}}.
\]

We conclude this section by specifying the constant $\hat\Theta$. First, define the complex number
\begin{equation}
Z_{[\omega]}:=\int_X\frac{(\alpha+\sqrt{-1}\omega)^n}{n!},
\end{equation}
which only depends on $[\alpha]$ and $\Omega$. Again, by specifying ${\rm Arg}$ as a map from $\mathbb{C} \rightarrow \mathbb{R}$ (rather than $S^1$), so that the argument of the constant function $1$ is zero, we define $\hat\Theta$ to be Arg$(Z_{[\omega]})$. Following \cite{JY} and \cite{WY}, we say that an angle is ${\it supercritical}$ if it is larger than $(n-2)\frac\pi2$, and ${\it hypercritical}$ if it is larger than $(n-1)\frac\pi2$. For further discussion and background we refer to reader to \cite{JY}.

\section{Subsolutions and the $C^{0}$ estimate}

 In order to introduce the notion of subsolution for the Lagrangian phase equation, we first define the relevant cone in which our solutions takes values.  Let $\Gamma_n \subset \mathbb{R}^{n}$ denote the positive orthant.  Recall that $(X,\alpha)$ is a fixed K\"ahler manifold, and $\omega_0$ is a fixed (1,1) form.  In this paper we are interested in finding forms $\omega$ such that $[\omega] = [\omega_0]$, and
 \begin{equation}\label{eq: SLag met}
 \Theta_{\alpha}(\omega) := \sum_{\ell=1}^{n} \arctan(\la_\ell) = h(x)
 \end{equation}
where $\la_\ell$ are the eigenvalues of the hermitian endomorphism $\Lambda^{i}_{k}:=\alpha^{i\bar{j}}\omega_{k\bar{j}}$, and $h: X\rightarrow ((n-2)\frac{\pi}{2}, n\frac{\pi}{2})$ is a smooth function.  We call this the Lagrangian phase equation with {\em supercritical phase}. To lighten notation, let us define $\Theta: \mathbb{R}^{n} \rightarrow \mathbb{R}$ to be
\[
\Theta(x_1, \dots ,x_n) = \sum_{i=1}^{n} \arctan(x_i).
\]
Let $\Gamma \subset \mathbb{R}^{n}$ be the cone through the origin over the region
 \[
 \left\{ (x_1, \dots, x_n) :  \Theta(x) \geq (n-2)\frac{\pi}{2} \right\}.
 \]
$\Gamma$ is an open, symmetric cone with vertex at the origin containing $\Gamma_n$.  Additionally, for any $\sigma \in((n-2)\frac{\pi}{2}, n\frac{\pi}{2})$ we define
\begin{equation}
\Gamma^{\sigma} := \{\lambda \in \Gamma: \Theta(\lambda) >\sigma \}.
\end{equation}
Note that for any $\sigma$ such that $\Gamma^{\sigma}$ is not empty, the boundary $\del \Gamma^{\sigma}$ is a smooth hypersurface.  The geometric and arithmetic properties of the cone $\Gamma$, and the sets $\Gamma^{\sigma}$ will play a crucial role in the developments to follow.  

\begin{lem}\label{lem: arithmetic}
Suppose we have real numbers $\la_1 \geq \la_2 \geq \cdots \geq \la_n$ which satisfy $\Theta(\lambda) = \sigma$,
for $\sigma \geq (n-2) \frac{\pi}{2}$.   Then $(\la_1, \dots,\la_n)$ have the following arithmetic properties;
\begin{enumerate}
\item[{\it (i)}] $\la_1\geq \la_2 \geq \cdots \geq \la_{n-1} >0$ and $\la_{n-1} \geq |\la_n|$.
\item[{\it (ii)}] $\la_1 +(n-1) \la_n \geq 0$.
\item[{\it (iii)}] $\sigma_k(\la_1,\dots,\la_n) \geq 0$ for all $1 \leq k \leq n-1$.
\end{enumerate}
Furthermore,
\begin{enumerate}
\item[{\it (iv)}] If $\Gamma^{\sigma}$ is not empty, the boundary $\del \Gamma^{\sigma}$ is a smooth, convex hypersurface. 
\end{enumerate}
In addition, if $\sigma \geq (n-2)\frac{\pi}{2} + \beta$, then;
\begin{enumerate}
\item[{\it (v)}] if $\lambda_n \leq 0$, then $\lambda_{n-1} \geq \epsilon_0(\beta)$.
\item[{\it (vi)}] $|\lambda_n| \leq C(\delta)$.
\end{enumerate}
\end{lem}
\begin{proof}
Statements {\it (i)-(iii)} are due to Wang-Yuan \cite[Lemma 2.1]{WY}.  Statement {\it(iv)} is Yuan \cite[Lemma 2.1]{Y}, while {\it (v)} and {\it (vi)} are trivial.
\end{proof}

In particular, it follows from part {\it (i)} of the above lemma that
\[
\Gamma \subset \{ (\la_1,\dots, \la_n) \in \mathbb{R}^{n} : \sum_{i} \la_i >0\}.
\]
We recall the definition of a $\mathcal{C}$-subsolution, due to Sz\'ekelyhidi \cite{Szek}.
 
\begin{defn}[\cite{Szek}, Definition 1]\label{def: C-sub}
Fix $\omega_0\in\Omega$. We say that a smooth function $\underbar{u}:X \rightarrow \mathbb{R}$ is a $\mathcal{C}$-subsolution of ~\eqref{eq: SLag met} if the following holds:  At each point $x\in X$ define the matrix $\Lambda^{i}_{j} := \alpha^{i\bar{k}}(\omega_0+\ddb \underbar{u})_{j\bar{k}}$.  Then we require that the set
\begin{equation} 
\left\{ \lambda' \in \Gamma : \sum_{\ell=1}^{n} \arctan(\la'_\ell) = h(x),\,\,\,\, {\rm and }\,\,\,\, \lambda' - \lambda(\Lambda(x)) \in \Gamma_n \right\}
\end{equation}
is bounded, where $\lambda(\Lambda(x))$ denotes the $n$-tuple of eigenvalues of $\Lambda(x)$.
\end{defn}

In the present setting we have the following explicit description of the $\mathcal{C}$-subsolutions. 

\begin{lem}\label{lem: sub sol def}
A smooth function $\underbar{u}:X \rightarrow \mathbb{R}$ is a $\mathcal{C}$-subsolution of ~\eqref{eq: SLag met} if and only if at each point $x \in X$, if $\mu_1, \dots, \mu_n$ denote the eigenvalues of the Hermitian endomorphism $\Lambda^{i}_{j} := \alpha^{i\bar{k}}(\omega_0+\ddb \underbar{u})_{j\bar{k}}$, then, for all $j=1, \dots,n$ we have
\begin{equation}
\label{subsolution1}
\sum_{\ell \ne j}  \arctan(\mu_\ell) > h(x) -\frac{\pi}{2}.
\end{equation}
\end{lem}
\begin{proof}
We show that if $\underline{u}$ satisfies \eqref{subsolution1}, then it is a $\mathcal{C}$-subsolution. Fix a point $x_0 \in X$, and  suppose we have numbers $\la_1\geq \la_2 \geq \dots\geq \la_n$ such that 
\[
\sum_{i} \arctan \la_i = h(x_0).
\]
It suffices to show that if $\la_i \geq \mu_i$ for all $i$, then $\la_1 \leq C$.  Fix $\delta >0$ such that
\[
\sum_{\ell \ne j}  \arctan(\mu_\ell) > h(x_0) +\delta -\frac{\pi}{2},
\]
and suppose we can find an $n$-tuple $\la_1\geq \la_2 \geq \dots\geq \la_n$ as above such that $\arctan \la_1 \geq \pi/2 -\delta$.  Then we clearly have
\[
\sum_{i\ne 1} \arctan(\la_i) \leq h(x_0) +\delta - \frac{\pi}{2} < \sum_{i\ne 1} \arctan(\mu_i).
\]
Since $\arctan(\cdot)$ is monotone increasing, we must have that $\mu_j >\la_j$ for some $j$, but this is a contradiction to the assumption that $\la_i \geq \mu_i$ for all $i$. The proof of the reverse implication is similar.
\end{proof}

Throughout this paper we will be somewhat abusive in referring to the $(1,1)$ form $\chi := \omega_0 +\ddb \underline{u}$ as a subsolution.  We hope that no confusion will result.

\begin{rk}\label{rk: form type pos}
The condition in Lemma~\ref{lem: sub sol def} can be expressed in terms of the positivity of a certain $(n-1,n-1)$ form, which is similar in spirit to the subsolution condition discovered by Song-Weinkove \cite{SW} in the setting of the $J$-flow.  We will discuss this fact, as well as some consequences in section~\ref{sec: stability}; see Proposition~\ref{prop: form type pos} below.
\end{rk}
The following proposition is due to Sz\'ekelyhidi \cite{Szek}, refining previous work of Guan \cite{Guan}. This proposition play a fundamental role in proving the $C^2$ bound for our equation, which we demonstrate in the next section.  

\begin{prop}[\cite{Szek}, Proposition 6]\label{prop: szek prop 6}
Let $[a,b] \subset ((n-2) \frac{\pi}{2}, n \frac{\pi}{2})$ and $\delta, R>0$.  There exists $\kappa>0$, with the following property:  Suppose that $\sigma \in [a,b]$ and $B$ is a hermitian matrix such that
\[
(\la(B)-2\delta Id + \Gamma_n) \cap \del \Gamma^{\sigma} \subset B_{R}(0).
\]
Then for any hermitian matrix $A$ with $\la(A) \in \del \Gamma^{\sigma}$ and $|\la(A)| > R$ we either have
\[
\sum_{p,q} \eta^{p\bar{q}}(A)[B_{p\bar{q}} - A_{p\bar{q}}] > \kappa \sum_p \eta^{p\bar{p}}(A)
\]
or $\eta^{i\bar{i}}(A) > \kappa \sum_p \eta^{p\bar{p}}(A)$ for all $i$, where $\eta = Id +A^2$.
\end{prop}
\begin{proof}
Since $\sigma > (n-2) \frac{\pi}{2}$, Lemma~\ref{lem: arithmetic} part {\it(iv)} implies that  $\Gamma^{\sigma}$ is a convex hypersurface.  With this observation, the proof in \cite{Szek} goes through verbatim.
\end{proof}

The following estimate, based on the Alexandroff-Bakelman-Pucci maximum principle, is due to Sz\'ekelyhidi \cite{Szek}.  B\l ocki  \cite{Bl} first applied the ABP estimate to the complex Monge-Amp\`ere equation on K\"ahler manifolds following earlier suggestions by Cheng and the third author.  While the operator under consideration here does not have the structural properties imposed in \cite{Szek}, it is straightforward to check that the proof requires only the ellipticity of the operator, and hence applies verbatim here.

\begin{prop}[\cite{Szek}, Proposition 10]\label{prop: C0 est}
Suppose that $\Theta_{\alpha}(\omega_0 + \ddb u) = h(x)$, where $h: X \rightarrow [(n-2) \frac{\pi}{2}, n\frac{\pi}{2})$, and suppose that $\underbar{u}=0$ is a $\mathcal{C}$-subsolution.  Then there exists a constant $C$, depending only on the given data, including $\omega_0$, such that
\[
\text{osc}_{X} u \leq C.
\]
\end{prop}

When a $\mathcal{C}$-subsolution $\underline{u}$ exists, we will denote by $\chi:=  \omega_0 +\ddb \underline{u} \in\Omega$ the corresponding $(1,1)$ form.

\section{The $C^{2}$ estimate}

The main result of this section is
\begin{thm}\label{thm: C2 estimate}
Suppose $u:X\rightarrow \mathbb{R}$ is a smooth function solving the equation
\begin{equation}\label{eq: C2 SLag}
\Theta_{\alpha}(\chi+\ddb u) = h(x)
\end{equation}
where $h: X \rightarrow [(n-2)\frac{\pi}{2}+\beta, n\frac{\pi}{2})$ for some $\beta>0$.  Then there exists a constant $C$ depending only on the subsolution $\chi$, as well as $|h|_{C^{2}(X,\alpha)}, {\rm osc}_X u, \alpha, \beta$, such that
\[
|\de\dbar u| \leq C\left(1+ \sup_{X}|\nabla u|^{2}\right).
\]
\end{thm}

\begin{proof}
The proof is via the maximum principle.  Let $\omega := \chi +\ddb u$.  We begin by defining functions $\phi(t), \psi(t)$ as follows.  Let $K =1+\sup_{X} |\nabla u|^{2}$, and set
\[
\phi(t) = -\frac{1}{2}\log(1- \frac{t}{2K}),\qquad t\in[0,K-1].
\]
Note that $\phi(t)$ satisfies 
\[
(4K)^{-1} < \phi' < (2K)^{-1}, \qquad \phi'' = 2(\phi')^2, \qquad  0\leq \phi(t) \leq \frac{1}{2}\log2
\]
Normalize $u$ so that $\inf_X u=0$.  By Proposition~\ref{prop: C0 est} we have a bound on $\sup_{X} u$.  Define $\psi:[0, \sup_{X}u] \rightarrow \mathbb{R}$ by   
\[
\psi(t) = -2At + \frac{A\tau}{2} t^{2}
\]
where $A \gg 0$ and $\tau>0$ are constants to be determined.  We choose $\tau$ sufficiently small so that
\[
A \leq -\psi' \leq 2A, \quad \psi'' = A\tau
\]
Define the Hermitian endomorphisms 
\[
\Lambda :=\alpha^{i\bar{j}}(\chi +\ddb u)_{k\bar{j}},\qquad \Lambda_0 := \alpha^{i\bar{j}}\chi_{k\bar{j}}
\] 
and recall that we are assuming $\chi$ is a $\mathcal{C}$-subsolution.  Let $\la_{max}$ denote the largest eigenvalue of $\Lambda$, which is a continuous function from $X$ to $\mathbb{R}$. We want to apply the maximum principle to the quantity
\[
G_{0}(x) := \frac{1}{2}\log(1+\lambda_{max}^{2}) + \phi(|\nabla u|^{2}) + \psi (u).
\]
This quantity is inspired by the one considered by Hou-Ma-Wu \cite{HMW} for the complex Hessian equations and subsequently used by Sz\'ekelyhidi \cite{Szek} for a large class of concave equations.  The gradient term used appearing in $G_0$ was first used by Chou-Wang \cite{CW} in their study of the real Hessian equations.  The function $G_{0}$ differs from the one considered in \cite{HMW, Szek} in its highest order term, where we have used a function of the eigenvalues which appears in the study of the real special Lagrangian equation; see, for instance \cite{WY, SmW}.  This modification is not merely cosmetic -- the added convexity of this higher order term appears essential to the estimate.  Finally, we note that,  unlike the estimates in the real case, we require extra lower order terms in order to counter additional negative terms which appear when differentiating the eigenvalues of a Hermitian (rather than symmetric) matrix.

The function $G_{0}(x)$ is clearly continuous, and hence achieves its maximum at some point $x_0 \in X$. Fix local coordinates $(z_1, \dots, z_n)$ centered at $x_0$ which are normal for the background K\"ahler metric $\alpha$, and such that $\omega(x_0)$ is diagonal with entries $\la_1 \geq \la_2 \geq \cdots \geq \la_n$.  By Lemma~\ref{lem: arithmetic} we may assume that $\la_1$ is sufficiently large so that $\la_1 >\max\{2|\la_n|, |\la_n|+1\}$. Since $\chi$ is a subsolution, we can find $\delta, R >0$ depending only on $h, \omega_0$ such that
\[
[\la(\Lambda_0(x_0)) - 2\delta Id + \Gamma_n] \cap \del \Gamma^{h(x_0)} \subset B_{R}(0).
\]
We may assume that $|\la(\Lambda(x_0))| >R$, so that Proposition~\ref{prop: szek prop 6} applies. In particular, there exists $\kappa>0$ depending only on $\delta, R$ and $h$ such that either
\begin{equation}\label{eq: C2 good pos}
\sum_p \frac{\chi_{p\bar{p}} - \la_p}{1+\la_p^2} > \kappa \sum_p \frac{1}{1+\la_p^2}.
\end{equation}
or $(1+\la_i^2)^{-1} > \kappa \sum_p (1+\la_p^2)^{-1}$ for all $i$.  Since $\la_n$ is uniformly bounded by Lemma~\ref{lem: arithmetic} part {\it (vi)}, we may assume that $\la_1$ is sufficiently large so that
\[
\frac{1}{1+\la_1^2} \leq \kappa\frac{1}{1+\la_n^2}.
\]
In particular,~\eqref{eq: C2 good pos} must hold.

In order to apply the maximum principle, we must differentiate the function $G_0$ twice.  Since the eigenvalues of $\Lambda$ need not be distinct at $x_0$, the function $G_0$ may only be continuous.  To circumvent this difficulty we use a perturbation argument similar to the one used in \cite{Szek}.   We choose a constant matrix $B$, defined in our fixed local coordinates to be a constant diagonal matrix $B_{pq}$ with real entries satisfying $B_{11}=B_{nn}=0$ and $0<B_{22} < \cdots < B_{n-1\,n-1}$, and such that
\[
\sum_j B_{jj} \leq (n-1)\frac{\epsilon_0}{2}
\]
where $\epsilon_0$ is the constant from Lemma~\ref{lem: arithmetic} part {\it(v)}. We work with the matrix $\tilde{\Lambda} = \Lambda-B$, and apply the maximum principle to the smooth function
\[
G(x) = \frac{1}{2}\log(1+\ti\lambda_{max}^{2}) + \phi(|\nabla u|^{2}) + \psi (u) 
\]
where $\ti\lambda_{max}$ denotes the largest eigenvalue of $\ti\Lambda$.  Note that $G(x) \leq G_0(x)$ and that $G(x)$ achieves its maximum at $x_0$, where we have $G(x_0) = G_0(x_0)$.  If we denote by $\ti{\lambda_i}$ are the eigenvalues of $\ti{\Lambda}$, then $\ti\lambda_1 = \lambda_1$, and all the remaining eigenvalues are distinct from $\ti\la_1$.   In particular, $\ti\la_1$ is a smooth function near $x_0$ and we may differentiate it freely.  Computing derivatives of $\ti\la_1$ yields 
\begin{equation}
\begin{aligned}
\nabla_s \ti\lambda_1 &= \nabla_s \omega_{1\bar{1}} - \nabla_sB_{11}\\
\nabla_{s}\nabla_{\bar{s}}\ti\la_1 &= \nabla_s \nabla_{\bar{s}} \omega_{1\bar{1}} + \sum_{q>1} \frac{ |\nabla_s\omega_{q\bar{1}}|^{2} + |\nabla_{s} \omega_{1\bar{q}}|^2}{(\lambda_1 - \ti\la_q)}\\
&\quad+ \nabla_s\nabla_{\bar{s}}B_{11} - 2{\rm Re} \sum_{q>1} \frac{ \nabla_s\omega_{q\bar{1}} \nabla_{\bar{s}}B_{1\bar{q}} + \nabla_{s}\omega_{1\bar{q}}\nabla_{\bar{s}}B_{p\bar{1}}}{\lambda_1 - \ti\la_p}\\
&\quad+ \ti\la_1^{pq,r\ell}\nabla_s(B_{pq})\nabla_{\bar{s}}(B_{r\ell}),
\end{aligned}
\end{equation}
where
\[
\ti\la_1^{pq,r\ell} = (1-\delta_{1p}) \frac{\delta_{1q}\delta_{1r}\delta_{p\ell}}{\ti\la_1-\ti\la_p} +(1-\delta_{1r}) \frac{\delta_{1\ell}\delta_{1p}\delta_{rq}}{\ti\la_1-\ti\la_r} 
\]
see, for example, \cite[Equation (70)]{Szek} or \cite{SzTW}.  Evaluating this expression at $x_0\in X$, and using that $B$ is constant, $B_{11}=0$, and that we are working in normal coordinates for $\alpha$, we have
\begin{equation}\label{eq: Lap EV0}
\begin{aligned}
\nabla_s \ti\lambda_1 &= \nabla_s \omega_{1\bar{1}}\\
\nabla_{s}\nabla_{\bar{s}}\ti\la_1 &= \nabla_s \nabla_{\bar{s}} \omega_{1\bar{1}} + \sum_{q>1} \frac{ |\nabla_s\omega_{q\bar{1}}|^{2} + |\nabla_{s} \omega_{1\bar{q}}|^2}{(\lambda_1 - \ti\la_q)}
\end{aligned}
\end{equation}
We are thus reduced to differentiating equation~\eqref{eq: C2 SLag}. Using \eqref{eq: derivtheta}, and computing at $x_0$, we have 
\begin{equation}\label{eq: Lap EV1}
\begin{aligned}
\nabla_{\bar{b}}\nabla_{a}h &= \nabla_{\bar b}(\eta^{s\bar{q}} \nabla_{a} \omega_{s\bar{q}}) \\
&= \eta^{s\bar{q}}\nabla_{\bar{b}} \nabla_{a} \omega_{s\bar{q}}+(\nabla_{\bar{b}}\eta^{s\bar{q}})\nabla_{a}\omega_{s\bar{q}} \\
&= \eta^{s\bar{q}}\nabla_s\nabla_{\bar{q}}\omega_{a\bar{b}} + \eta^{s\bar{q}}[\nabla_{\bar{b}},\nabla_{s}]\omega_{a\bar{q}} + (\nabla_{\bar{b}}\eta^{s\bar{q}})\nabla_{a}\omega_{s\bar{q}}.
\end{aligned}
\end{equation}
Expanding the third term, we have
\begin{equation}
\begin{aligned}
(\nabla_{\bar{b}}\eta^{s\bar{q}}) &= - \eta^{s\bar{k}}\eta^{j\bar{q}}\nabla_{\bar{b}}\eta_{j\bar{k}}\\
&= - \eta^{s\bar{k}}\eta^{j\bar{q}}\left(\alpha^{p\bar{m}}\omega_{p\bar{k}}\nabla_{\bar{b}}\omega_{j\bar{m}} + \alpha^{p\bar{m}}\omega_{j\bar{m}}\nabla_{\bar{b}}\omega_{p\bar{k}}\right),
\end{aligned}
\end{equation}
Using that $\alpha, \omega$ are diagonal at $x_0$, we can now solve for $\Delta_\eta\omega_{1\bar{1}}$, 
\begin{equation}
\begin{aligned}
\eta^{s\bar{q}}\nabla_s\nabla_{\bar{q}} \omega_{1\bar{1}} &= \nabla_{1}\nabla_{\bar{1}}h - \sum_{s} \frac{\la_s R_{1}\,^{s}\,_{s\bar{1}}}{1+\la_s^2} \\
&\quad+ \sum_{s} \frac{\la_1R^{\bar{1}}\,_{\bar{s}s\bar{1}}}{1+\la_s^2} + \sum_{s,q} \frac{\la_s+\la_q}{(1+\la_s^2)(1+\la_q^2)}|\nabla_1\omega_{q\bar{s}}|^{2}.
\end{aligned}
\end{equation}
Combining this expression with~\eqref{eq: Lap EV0} allows us to solve for $\Delta_\eta\ti\lambda_1$. This allows us to compute the Laplacian of the highest order term from $G(x)$ at the point $x_0$
\begin{equation*}
\begin{aligned}
\Delta_{\eta} \frac{1}{2}\log(1+\ti\la_{1}^2)  &= \frac{\lambda_1\Delta_{\eta}\ti\lambda_1}{(1+\lambda_1^2)} + \frac{1-\lambda_1^2}{(1+\lambda_1^2)^3}|\nabla_1\omega_{1\bar1}|^2\\
&\quad +\sum_{s>1}\frac{1-\lambda_1^2}{(1+\lambda_1^2)^2(1+\lambda_s^2)}|\nabla_s\omega_{1\bar 1}|^2. 
\end{aligned}
\end{equation*}
The Laplacian term can be computed as
\begin{equation}
\begin{aligned}
\Delta_{\eta}\lambda_1&=\nabla_{1}\nabla_{\bar{1}}h - \sum_{s} \frac{\la_s R_{1}\,^{s}\,_{s\bar{1}}}{1+\la_s^2}+ \sum_{s} \frac{\la_1R^{\bar{1}}\,_{\bar{s}s\bar{1}}}{1+\la_s^2}+\sum_s\frac{\nabla_s\nabla_{\bar{s}}B_{11}}{1+\lambda_s^2}\nonumber\\
&\quad +  \sum_{s,q} \frac{\la_s+\la_q}{(1+\la_s^2)(1+\la_q^2)}|\nabla_1\omega_{q\bar{s}}|^{2}+ \sum_s\sum_{q>1} \frac{ |\nabla_s\omega_{q\bar{1}}|^{2} + |\nabla_{s} \omega_{1\bar{q}}|^2}{(1+\lambda_s^2)(\lambda_1 - \ti\la_q)}.
\end{aligned}
\end{equation}

After some algebra we arrive at the formula
\begin{equation}
\begin{aligned}
\label{eq: C2 equality}
\Delta_{\eta} &\frac{1}{2}\log(1+\ti\la_1^2) = \frac{\la_1}{1+\la_1^2} \nabla_1\nabla_{\bar{1}}h +\sum_{s}\frac{-\la_1 \la_s R_{1}\,^{s}\,_{s\bar{1}}+\la_1^2 R^{\bar{1}}\,_{\bar{s}s\bar{1}}}{(1+\la_1^2)(1+\la_s^2)}\\
&\quad +  \sum_s\sum_{q>1} \frac{\la_1[1+\la_1(\la_s+\la_q)-\la_q\la_s + (\la_s+\la_q)(\la_q-\ti\la_q)]}{(1+\la_1^2)(1+\la_s^2)(1+\la_q^2)(\la_1-\ti\la_q)} |\nabla_{1}\omega_{q\bar{s}}|^{2} \\
&\quad+\frac{1}{(1+\la_1^2)^2}|\nabla_1\omega_{1\bar{1}}|^2 + \sum_{s>1} \frac{\lambda_1^2\lambda_s + 2\lambda_1 - \ti\la_s + \la_s(\la_s-\ti\la_s)}{(1+\la_1^2)^2(1+\la_s^2)(\la_1-\ti\la_s)} |\nabla_s\omega_{1\bar{1}}|^{2}\\
&\quad + \sum_{s,q>1} \frac{\la_1}{(1+\la_1^2)(1+\la_s^2)(\la_1-\ti\la_q)}|\nabla_q\omega_{s\bar{1}}|^2.
\end{aligned}
\end{equation}

The main difficulty is finding a useful estimate for this quantity.  For the remainder of this section we let $C$ denote a constant depending only on the stated data, but which may change from line to line.  The first two terms contribute only a negative constant.  For the third term, we require the following simple lemma
\begin{lem}
If $\lambda_1 \geq \lambda_2 \geq \cdots \geq \lambda_n$, and these numbers satisfy $\Theta(\la) \geq (n-2)\frac{\pi}{2}$, then
\[
1+\lambda_1(\lambda_j + \lambda_\ell) - \lambda_j\lambda_\ell \geq0 
\]
unless $j=\ell=n$ and $\lambda_n < 0$.
\end{lem}
\begin{proof}
The lemma is obvious if $\lambda_j, \lambda_\ell \geq 0$, since $\lambda_1 \geq \max\{ \lambda_j, \lambda_\ell\}$.  By symmetry we can consider the case when $j=n, \lambda_n <0$, and $\ell <n$.  In this case Lemma~\ref{lem: arithmetic} part {\it (i)} guarantees that $\lambda_\ell + \lambda_1 \geq 0$, and so again we are done, since the final term above is positive.  
\end{proof}

The fourth term in~\eqref{eq: C2 equality} is positive, as is the fifth term, unless $s=n$ and $\lambda_n<0$.  The sixth term is also clearly positive.  Thus, if $\lambda_n\geq0$, then
\begin{equation}
\Delta_{\eta} \frac{1}{2}\log(1+\ti\la_1^2) \geq -C.\nonumber
\end{equation}
If $\lambda_n<0$ then the estimate is much worse, due to the presence of several negative terms. Throwing away some but not all of the positive terms,  we rewrite \eqref{eq: C2 equality} as
\begin{equation}
\begin{aligned}
\label{hardercase}
\Delta_{\eta} \frac{1}{2}\log(1+\ti\la_1^2) &\geq -C +\frac{1}{(1+\la_1^2)^2}|\nabla_1\omega_{1\bar{1}}|^2 \\
&\quad+  \sum_{q>1} \frac{\la_1[1+2\la_1\la_q-\la_q^2 + 2\la_q(\la_q-\ti\la_q)]}{(1+\la_1^2)(1+\la_q^2)^2(\la_1-\ti\la_q)} |\nabla_{1}\omega_{q\bar{q}}|^{2} \\ 
&\quad + \sum_{q>1} \frac{\lambda_1^2\lambda_q + 2\lambda_1 - \ti\la_q + \la_q(\la_q-\ti\la_q)}{(1+\la_1^2)^2(1+\la_q^2)(\la_1-\ti\la_q)} |\nabla_q\omega_{1\bar{1}}|^{2}.\\
\end{aligned}
\end{equation}
Let us analyze this more difficult case.  We first estimate the second line above.  Note that we can write
\begin{equation*}
\begin{aligned}
\frac{\la_1[1+2\la_1\la_q-\la_q^2 + 2\la_q(\la_q-\ti\la_q)]}{(1+\la_1^2)(1+\la_q^2)^2(\la_1-\ti\la_q)}  &= \frac{\lambda_q}{(1+\la_q^2)^2(\la_1-\ti\la_q)}\\
&\quad + \frac{(\la_1-\la_q)(1+\la_1\la_q)+2\la_q\la_1(\la_q-\ti\la_q)}{(1+\la_1^2)(1+\la_q^2)^2(\la_1-\ti\la_q)}\\
&= \frac{\lambda_q}{(1+\la_q^2)^2(\la_1-\ti\la_q)} + \frac{1+\la_1\la_q}{(1+\la_1^2)(1+\la_q^2)^2}\\
&\quad + \frac{(\la_q\la_1-1)(\la_q-\ti\la_q)}{(1+\la_1^2)(1+\la_q^2)^2(\la_1-\ti\la_q)},
\end{aligned}
\end{equation*}
and so we can rewrite the first and second lines in \eqref{hardercase} (excluding the constant) as three separate sums
\begin{equation}\label{eq: C2 est terms}
\begin{aligned}
\big( \text{{\bf I}} \big) &= \sum_{q>1} \frac{\lambda_q}{(1+\la_q^2)^2(\la_1-\ti\la_q)} |\nabla_{1}\omega_{q\bar{q}}|^{2} + \frac{1}{(1+\la_1^2)^2} |\nabla_1 \omega_{1\bar{1}}|^{2}\\
\big( \text{{\bf II}} \big) &= \sum_{q} \frac{1+\la_1\la_q}{(1+\la_1^2)(1+\la_q^2)^2} |\nabla_{1}\omega_{q\bar{q}}|^{2}- \frac{1}{(1+\la_1^2)^2} |\nabla_1 \omega_{1\bar{1}}|^{2}\\
\big( \text{{\bf III}} \big) &= \sum_{1<q<n}\frac{(\la_q\la_1-1)(\la_q-\ti\la_q)}{(1+\la_1^2)(1+\la_q^2)^2(\la_1-\ti\la_q)} |\nabla_{1}\omega_{q\bar{q}}|^{2}
\end{aligned}
\end{equation}

where we have used that $B_{nn}=0$.  We may assume that $\la_1$ is sufficiently large so that for $q<n$ we have $\la_1\la_q \geq 1$, since $\la_{n-1} \geq \epsilon_0$ by Lemma~\ref{lem: arithmetic} part {\it (v)}.  In particular, the third sum is positive.  We next consider terms ({\bf I}) and ({\bf II}) individually, beginning with term ({\bf I}).  The only negative contribution to the sum occurs when $q=n$.  Differentiating our main equation \eqref{eq: C2 SLag}, we have, for any $\delta, \alpha_j >0$, $j=1,\dots,n-1$
\begin{equation}\label{eq: delt split up}
\begin{aligned}
\frac{|\nabla_1\omega_{n\bar{n}}|^{2}}{(1+\la_n^2)^2} &= \bigg| \nabla_1h - \sum_{q<n} \frac{\nabla_1\omega_{q\bar{q}}}{1+\la_q^2}\bigg|^2\\
&\leq (1+\frac{\la_1}{\delta})|\nabla_1h|^{2} + (1+\frac{\delta}{\la_1})\bigg| \sum_{q<n} \frac{\nabla_1\omega_{q\bar{q}}}{1+\la_q^2}\bigg|^2\\
&\leq (1+\frac{\la_1}{\delta})|\nabla_1h|^{2} + (1+\frac{\delta}{\la_1})\left( \sum_{q<n} \frac{|\nabla_1\omega_{q\bar{q}}|^{2}\alpha_{q}}{(1+\la_q^2)^2}\right) \cdot \left(\sum_{j<n} \frac{1}{\alpha_{j}}\right).
\end{aligned}
\end{equation}
In the above we have used Young's inequality for the first line, and the Cauchy-Schwartz inequality in the third line. Now, set $\alpha_q = \frac{\lambda_q}{\la_1-\ti\la_q}$ for $1<q<n$, and $\alpha_1=1$, and choose $\delta = \epsilon_0/2$, where $\epsilon_0$ is as in Lemma~\ref{lem: arithmetic} part {\it (v)}.  Let us denote
\begin{equation}
\begin{aligned}
\Upsilon &:= \left(\sum_{q<n} \frac{|\nabla_1\omega_{q\bar{q}}|^{2}\alpha_{q}}{(1+\la_q^2)^2}\right)\\
&\quad =  \sum_{1<q<n} \frac{\lambda_q}{(1+\la_q^2)^2(\la_1-\ti\la_q)} |\nabla_{1}\omega_{q\bar{q}}|^{2}+ \frac{|\nabla_1 \omega_{1\bar{1}}|^{2}}{(1+\la_1^2)}.\nonumber
\end{aligned}
\end{equation}
Multiplying \eqref{eq: delt split up} by $\frac{\lambda_n}{(\lambda_1-\lambda_n)}$ and observing that $\frac{\lambda_n(\delta+\lambda_1)|\nabla_1h|^2}{\delta(\lambda_1-\lambda_n)}\geq-C$, by our choice of $\alpha_j$ we have,
\begin{equation}
\frac{\la_n|\nabla_1\omega_{n\bar{n}}|^{2}}{(1+\la_n^2)^2(\la_1-\la_n)}  \geq -C + \frac{\la_n}{(\la_1-\la_n)}(1+\frac{\delta}{\la_1})\left(1 +\sum_{1<j<n} \frac{\la_1 - \ti\la_j}{\la_j}\right)\Upsilon.
\end{equation}
Note that the left hand side above is the $q=n$ term from ({\bf I}), while the remaining terms from ({\bf I}) are equal to $\Upsilon$. Using that $\ti\la_n = \la_n<0$, we estimate ({\bf I}) as follows
\begin{equation}
\begin{aligned}
\text{ ({\bf I})}&\geq -C + \Upsilon \frac{\la_n}{\la_1-\la_n}\left\{\frac{\la_1-\la_n}{\la_n} +\sum_{j<n} \frac{\la_1}{\la_j} + \delta \sum_{j<n}\frac{1}{\la_j} -(1+ \frac{\delta}{\la_1}) \sum_{1<j<n} \frac{\ti\la_j}{\la_j}\right\}\\
&\geq -C + \Upsilon \frac{\la_n}{\la_1-\la_n}\left\{ \la_1\sum_{j}\frac{1}{\la_j} - \sum_{j>1}\frac{\ti\la_j}{\la_j} + \delta \sum_{j<n}\frac{1}{\la_{j}} \right\}\\
&\geq -C +\Upsilon \frac{\la_n}{\la_1-\la_n}\left\{ \la_1\frac{\sigma_{n-1}(\lambda)}{\sigma_n(\lambda)}- \sum_{j>1}\frac{\ti\la_j}{\la_j} + \delta \frac{n-1}{\la_{n-1}} \right\}.
\end{aligned}
\end{equation}
Since $\sigma_{n-1}(\la(\Lambda)) \geq 0$, and $\sigma_n(\la(\Lambda))<0$ by Lemma~\ref{lem: arithmetic} part {\it (iii)}, the first term in the brackets is negative.  Furthermore, by our choice of $B$ we know that
\[
\begin{aligned}
\sum_{j>1} \frac{\ti\la_j}{\la_j} &= (n-1) -\sum_{1<j<n} \frac{B_{jj}}{\la_j}\\
& \geq (n-1) -\frac{1}{\epsilon_0} \sum_j B_{jj}\\
&\geq \frac{n-1}{2},
\end{aligned}
\]
and hence our choice of $\delta$ implies that the final two terms combine to be negative as well. Thus, we obtain that the term ({\bf I}) in equation~\eqref{eq: C2 est terms} is bounded below by a negative constant depending only on the stated data.

Next we consider the sign of the sum ({\bf II}). Again, the only negative contribution to the sum occurs when $q=n$.

We use an estimate similar to that in~\eqref{eq: delt split up} to get that, for any $\delta, \alpha_j, \alpha_j'>0$, $1\leq j<n$
\begin{equation}\label{eq: C2 II est1}
\begin{aligned}
\frac{\la_1\la_n|\nabla_1\omega_{n\bar{n}}|^{2}}{(1+\la_1^2)(1+\la_n^2)^2}  &\geq -\frac{C}{\delta} + \frac{\la_1\la_n}{(1+\la_1^2)}(1+\frac{\delta}{\la_1})\bigg| \sum_{q<n} \frac{\nabla_1\omega_{q\bar{q}}}{1+\la_q^2}\bigg|^2\\
&\geq -\frac{C}{\delta} + \la_n\left( \sum_{q<n} \frac{|\nabla_1\omega_{q\bar{q}}|^{2}\la_1\alpha_{q}}{(1+\la_1^2)(1+\la_q^2)^2}\right) \cdot \left(\sum_{j<n} \frac{1}{\alpha_{j}}\right)\\
&\quad +\delta \la_n\left( \sum_{q<n} \frac{|\nabla_1\omega_{q\bar{q}}|^{2}\alpha_{q}'}{(1+\la_1^2)(1+\la_q^2)^2}\right) \cdot \left(\sum_{j<n} \frac{1}{\alpha_{j}'}\right)
\end{aligned}
\end{equation}
where in the last line we have used the Cauchy-Schwartz inequality twice.  We take $\alpha_q =\la_q$, and $\alpha_q'=1$ for $1\leq q<n$.  To simplify notation, let us define
\[
\ti\Upsilon = \sum_{q<n}\frac{|\nabla_1\omega_{q\bar{q}}|^{2}\la_1\la_{q}}{(1+\la_1^2)(1+\la_q^2)^2}.
\]
Substituting the estimate in~\eqref{eq: C2 II est1} into the expression for term ({\bf II}) and simplifying we obtain
\begin{equation}
\begin{aligned}
\big(\text{{\bf II}}\big) &\geq -\frac{C}{\delta} - \frac{1}{(1+\la_1^2)^2} |\nabla_1 \omega_{1\bar{1}}|^{2}+ \frac{|\nabla_1 \omega_{n\bar{n}}|^{2}}{(1+\la_1^2)(1+\la_n^2)^2} \\
&\quad + \bigg\{1 +\delta(n-1) \la_n\bigg\} \sum_{q<n} \frac{|\nabla_1 \omega_{q\bar{q}}|^{2}}{(1+\la_1^2)(1+\la_q^2)^2}  + \ti\Upsilon \left\{ 1+ \la_n \left(\sum_{j<n} \frac{1}{\la_j}\right)\right\}.
\end{aligned}
\end{equation}
If we choose $\delta$ sufficiently small depending only on the uniform lower bound for $\la_n$ provided by Lemma~\ref{lem: arithmetic} part {\it (vi)}  then the first term on the second line is positive, while the final term is always positive by Lemma~\ref{lem: arithmetic} part {\it (iii)}. Thus
\[
\big( \text{{\bf II}} \big) \geq -C- \frac{1}{(1+\la_1^2)^2} |\nabla_1 \omega_{1\bar{1}}|^{2}. 
\]
Summarizing, we have proven the estimate
\begin{equation}\label{eq: C2 est key}
\begin{aligned}
\Delta_{\eta} \frac{1}{2}\log(1+\ti\la_1^2) &\geq -C -\frac{1}{(1+\la_1^2)^2}|\nabla_1\omega_{1\bar{1}}|^2 \\
&\quad + \sum_{q>1} \frac{\lambda_1^2\lambda_q + 2\lambda_1 - \ti\la_q + \la_q(\la_q-\ti\la_q)}{(1+\la_1^2)^2(1+\la_q^2)(\la_1-\ti\la_q)} |\nabla_q\omega_{1\bar{1}}|^{2}\\
&\geq -C -\frac{1}{(1+\la_1^2)^2}|\nabla_1\omega_{1\bar{1}}|^2 +\frac{\la_1^2\la_n|\nabla_n \omega_{1\bar{1}}|^{2}}{(1+\la_1^2)^2(1+\la_n^2)(\la_1-\la_n)}.
\end{aligned}
\end{equation}
where in the last line we have used the obvious fact that
\[
\lambda_1^2\lambda_q + 2\lambda_1 - \ti\la_q + \la_q(\la_q-\ti\la_q)\geq 0, \qquad 1<q<n.
\]
We now compute the action of the linearized operator on the lower order terms in the definition of $G$.
\begin{equation*}
\begin{aligned}
\Delta_{\eta} \psi(u) &= \psi''(u) \sum_{q}\frac{|u_q|^{2} }{1+\la_q^2} + \psi'(u) \sum_{q}\frac{\la_q - \chi_{q\bar{q}}}{1+\la_q^2}\\
\Delta_{\eta} \phi(|\nabla u|^{2}) &= \frac{\phi''(|\nabla u|^{2})}{1+\la_q^2} \sum_{q}\left|\sum_{j} u_{q\bar{j}}u_{j} +  u_{qj}u_{\bar{j}}\right|^{2}\\
&\quad +2\phi'(|\nabla u|^{2})\sum_{j}\text{Re}\left( u_{j}h_{\bar{j}} - \sum_{q}\frac{u_{j}\nabla_{\bar{j}}\chi_{q\bar{q}}}{1+\la_q^2}\right)\\
&\quad + \phi'(|\nabla u|^{2})\sum_{q}\left(\sum_{j} \frac{|u_{j\bar{q}}|^{2} + |u_{qj}|^{2}}{1+\la_q^2} + \frac{R_{q\bar{q}}\,^{\bar{k}\ell}u_{\ell}u_{\bar{k}}}{1+\la_q^2}\right).\\
\end{aligned}
\end{equation*}

Now, it is easy to see that

\[
 \frac{R_{q\bar{q}}\,^{\bar{k}\ell}u_{\ell}u_{\bar{k}}}{1+\la_q^2}+2\text{Re}\left( u_{j}h_{\bar{j}} - \frac{u_{j}\nabla_{\bar{j}}\chi_{p\bar{p}}}{1+\la_p^2}\right) \geq -C_0K,
\]
so at $x_0$, where $G$ achieves its maximum, we have 
\begin{equation*}
\begin{aligned}
0 \geq \Delta_{\eta} G &\geq -C_1 -\frac{|\nabla_1\omega_{1\bar{1}}|^2}{(1+\la_1^2)^2} +\frac{\la_1^2\la_n|\nabla_n \omega_{1\bar{1}}|^{2}}{(1+\la_1^2)^2(1+\la_n^2)(\la_1-\la_n)}\\
&\quad +\phi'' \sum_q \frac{1}{1+\la_q^2}\left|\sum_{j} u_{q\bar{j}}u_{j} + u_{qj}u_{\bar{j}}\right|^{2} -\phi'C_0K\\
&\quad + \sum_q\left(\phi'\sum_{j} \frac{|u_{j\bar{q}}|^{2} + |u_{qj}|^{2}}{1+\la_q^2}+ \psi''\frac{|u_q|^2}{1+\la_q^2} + \psi'\frac{\la_q - \chi_{q\bar{q}}}{1+\la_q^2}\right).
\end{aligned}
\end{equation*}
Furthermore, we have $\nabla_pG(x_0)=0$, and so
\begin{equation*}
\frac{\lambda_1\nabla_p \omega_{1\bar{1}}}{1+\la_1^2} = -\phi'\sum_{j}\left(u_{pj}u_{\bar{j}} + u_{j}u_{p\bar{j}}\right) - \psi'u_{p}
\end{equation*}
In particular,
\begin{equation*}
\begin{aligned}
\frac{\la_1^2\la_n|\nabla_n \omega_{1\bar{1}}|^{2}}{(1+\la_1^2)^2(1+\la_n^2)(\la_1-\la_n)} &\geq \frac{\la_n(1+\delta)(\phi')^2}{(1+\la_n^2)(\la_1-\la_n)}\left|\sum_{j}u_{nj}u_{\bar{j}} + u_{j}u_{n\bar{j}}\right|^{2}\\
&\quad + \frac{\la_n(1+\delta^{-1})(\psi')^2}{(1+\la_n^2)(\la_1-\la_n)}|u_{n}|^{2}.
\end{aligned}
\end{equation*}
In a similar fashion we have
\begin{equation*}
\begin{aligned}
\frac{|\nabla_1\omega_{1\bar{1}}|^2}{(1+\la_1^2)^2} &\leq \frac{(1+\delta)(\phi')^2}{\la_1^2}\left|\sum_{j}u_{1j}u_{\bar{j}} + u_{j}u_{1\bar{j}}\right|^{2}\\
&\quad + \frac{(1+\delta^{-1})(\psi')^2}{\la_1^2}|u_{1}|^{2}.
\end{aligned}
\end{equation*}
We now use that $\phi'' = 2(\phi')^2$.  If we take $\delta =1/2$, then we have at $x_0$
\begin{equation*}
\begin{aligned}
0 &\geq -C_1 + (\phi')^2\left|\sum_{j} u_{1\bar{j}}u_{j} + u_{1j}u_{\bar{j}}\right|^{2}  \left(\frac{2}{1+\la_1^2}-\frac{1+\delta}{\la_1^2}\right) \\
&\quad+\frac{(\phi')^2}{1+\la_n^2}\left|\sum_{j} u_{n\bar{j}}u_{j} + u_{nj}u_{\bar{j}}\right|^{2}\left(2 +\frac{\la_n(1+\delta)}{\la_1-\la_n}\right)\\   
&\quad- \frac{(1+\delta^{-1})(\psi')^2}{\la_1^2}|u_{1}|^{2}+\frac{\la_n(1+\delta^{-1})(\psi')^2}{(1+\la_n^2)(\la_1-\la_n)}|u_{n}|^{2}\\
&\quad + \phi''\sum_{1<q<n}\frac{1}{1+\la_q^2} \left|\sum_{j} u_{q\bar{j}}u_{j} + u_{qj}u_{\bar{j}}\right|^{2} -\phi'C_0K\\
&\quad + \sum_q\left(\phi'\sum_{j} \frac{|u_{j\bar{q}}|^{2} + |u_{qj}|^{2}}{1+\la_q^2}+ \psi''\frac{|u_q|^2}{1+\la_q^2} + \psi'\frac{\la_q - \chi_{q\bar{q}}}{1+\la_q^2}\right).
\end{aligned}
\end{equation*}
If $\la_1$ is sufficiently large, depending only on the lower bound for $\la_n$, then 
\[
\left(\frac{2}{1+\la_1^2}-\frac{1+\delta}{\la_1^2}\right) \geq 0,  \qquad \left(2 +\frac{\la_n(1+\delta)}{\la_1-\la_n}\right)\geq 0.
\]
In particular, since $\phi'' \geq 0$ we have
\begin{equation*}
\begin{aligned}
0 &\geq -C_1 -\phi'C_0K- \frac{(1+\delta^{-1})(\psi')^2}{\la_1^2}|u_{1}|^{2}\\
&\quad +\left(\psi''+\frac{\la_n(1+\delta^{-1})(\psi')^2}{\la_1-\la_n}\right)\frac{|u_n|^2}{1+\la_n^2}\\
&\quad + \phi'\sum_{q,j} \frac{|u_{j\bar{q}}|^{2} + |u_{qj}|^{2}}{1+\la_q^2}+ \psi''\sum_{q<n}\frac{|u_q|^2}{1+\la_q^2} + \psi'\sum_q\frac{\la_q - \chi_{q\bar{q}}}{1+\la_q^2}.
\end{aligned}
\end{equation*}
As long as $\la_1$ is sufficiently large, depending only on $\tau, A$, the bracketed term on the second line is positive. For the last term on the first line we clearly have the estimate
\[
\frac{(1+\delta^{-1})(\psi')^2}{\la_1^2}|u_1|^2 \leq \frac{C_0A^2K}{\la_1^2},
\]
and so
\begin{equation*}
\begin{aligned}
0 &\geq -C_1 -\phi'C_0K - \frac{C_0A^2K}{\la_1^2}\\
&\quad + \phi'\sum_{q,j} \frac{|u_{j\bar{q}}|^{2} + |u_{qj}|^{2}}{1+\la_p^2}+  \psi'\sum_{q}\frac{\la_q - \chi_{q\bar{q}}}{1+\la_q^2}.
\end{aligned}
\end{equation*}
Now, if $\la_1$ is sufficiently large relative to $\chi_{1\bar{1}}$, then we have
\[
|u_{1\bar{1}}|^{2} \geq \frac{1}{2} \la_1^2
\]
and so, since $(4K)^{-1} < \phi' < (2K)^{-1}$ we have 
\begin{equation}
\begin{aligned}
0 &\geq -C_1 -C_0 - \frac{C_0A^2K}{\la_1^2}\\
&\quad + \frac{1}{2}\phi'\frac{\la_1^{2}}{1+\la_1^2}+  \psi'\sum_{q}\frac{\la_q - \chi_{q\bar{q}}}{1+\la_q^2}.
\end{aligned}
\end{equation}
Recall from equation~\eqref{eq: C2 good pos} that we have
\[
\sum_{q}\frac{\chi_{q\bar{q}}-\la_q}{1+\la_q^2} \geq \kappa \sum_{q} \frac{1}{1+\la_q^2}. 
\]
Since $|\la_n| \leq C_3$ by Lemma~\ref{lem: arithmetic}, we can choose $A$ sufficiently large so that
\[
A\frac{\kappa}{1+\la_n^2} \geq -C_1-C_0,
\]
then, since $A<-\psi'<2A$, we have
\[
0 \geq  \frac{\la_1^2}{8K(1+\la_1^2)} - \frac{C_0A^2K}{\la_1^2}.
\]
In other words,
\[
 \frac{\la_1^2}{K^2} \leq \frac{8C_{0}A ^2(1+\la_1^2)}{\la_1^2} \leq C_5.
\]
Thus, at the maximum of $G$ we have $\la_1 \leq C_5K$. At this point we have
\[
\frac{1}{2}\log(1+\la_1^2)  - \frac{1}{2}\log(1-\frac{|\nabla u|^2}{2K}) + \psi(u) \leq \frac{1}{2}\log(1+C_5K^2) +C
\]
which after simplification yields the desired estimate;
\[
\sqrt{1+\la_1^2} \leq C_6 K.
\]
\end{proof}

\section{The blow-up argument and the gradient estimate}

We now apply a blow-up argument to the estimate in Theorem~\ref{thm: C2 estimate} to obtain a gradient bound.  By contrast with the general setting considered by Sz\'kelyhidi \cite{Szek}, or the complex Hessian equation studied by Dinew-Ko\l odziej \cite{DK}, the argument here is rather simple.  By the lower bound for $\omega$ from Lemma~\ref{lem: arithmetic}, part {\it (vi)} it suffices to prove

\begin{prop}
Suppose $u: X \rightarrow \mathbb{R}$ satisfies
\begin{itemize}
\item[{\it (i)}] $\omega_0+\ddb u \geq -K\alpha$,
\item[{\it (ii)}] $\sup_{X}|u| \leq K$,
\item[{\it (iii)}] $|\de\dbar u| \leq K(1+\sup_{X} |\nabla u|^{2})$,
\end{itemize}
for a uniform constant $K<+\infty$.  Then there exists a constant $C$, depending only on $(X,\alpha), \omega_0,$ and $K$ such that
\[
\sup_{X} |\nabla u| \leq C.
\]
\end{prop}
\begin{proof}
We argue by contradiction.  Suppose we have a K\"ahler manifold $(X,\alpha)$ where the estimate fails.  Then we have smooth functions $u_{k}:X\rightarrow \mathbb{R}$, and a $(1,1)$ form $\omega_0$ such that the assumptions ${\it (i)-(iii)}$ hold uniformly in $n\in \mathbb{N}$, but 
\[
\sup_{X} |\nabla u_k| = C_k \geq k.
\]
Let $x_{k} \in X$ be a point where $\sup_{X} |\nabla u_k|$ is attained.  Up to passing to a subsequence we may assume that $\{x_{k}\}$ converges to some point $x\in X$.  In particular, we may assume that about each $x_k$ there is a coordinate chart $U_k \subset X$ with coordinates $(z_1,\dots,z_n)$ defined on a ball of radius $1$, centered at $x_k$, such that
\[
 \alpha(z) = Id + O(|z|^2)
 \]
 on $U_k$.  In particular, estimates {\it (i)-(iii)} hold uniformly on $B_1(0)$ with $\alpha$ replaced by the Euclidean metric, after possibly increasing $K$ slightly.  Define $\hat{u}_{k}(z) := u_{k}(\frac{z}{C_k})$, defined in the ball of radius $C_k$.  From properties {\it (i)-(iii)} and the above remark we have
\begin{itemize}
\item $\de\dbar\hat{u}_{k} \geq \frac{-KId-\omega_0}{C_k^{2}}$ for all $z\in B_{C_k}(0)$,
\item ${\rm osc}_{B_{C_k}(0)} \hat{u}_{k} \leq K$,
\item $|\de\dbar \hat{u}_{k}| \leq 2K$ for all $z \in B_{C_k}(0)$
\item $|\nabla \hat{u}_{k}(z)| \leq 1 = |\nabla \hat{u}_{k}(0)|$ for all $z \in B_{C_k}(0)$.
\end{itemize}
Since $C_{k} \rightarrow \infty$, a standard diagonal argument yields, for a fixed $\beta \in(0,1)$, the existence of a $C^{1,\beta}$ function $u_{\infty}:\mathbb{C}^{n} \rightarrow \mathbb{R}$ so that $\hat{u}_{j} \rightarrow u_{\infty}$ in $C^{1,\beta}$ topology on compact subsets.  Furthermore, by the above estimates $u$ is continuous, uniformly bounded, has $|\nabla u(0)| =1$, and satisfies $\ddb u \geq 0$ in the sense of distributions.  Hence, $u$ is bounded, non-constant plurisubharmonic function defined on all of $\mathbb{C}^{n}$.  By a standard result in several complex variables, no such functions exist \cite{Ron}.
\end{proof}

\section{Higher order estimates}

The higher order estimates follow from the Evans-Krylov theory.  The equation~\eqref{eq: Spec Lag Ang} is only concave when $h: X \rightarrow [(n-1) \frac{\pi}{2}, n\frac{\pi}{2})$, the so called {\em hypercritical phase} case.  However, as long as $h \geq (n-2) \frac{\pi}{2}$, we can exploit the convexity of the level sets $\del \Gamma^{\sigma}$ (see Lemma~\ref{lem: arithmetic} part {\it (iv)}) to obtain the $C^{2,\beta}$ estimates by a blow-up argument.  The first step in this direction is to prove a Louiville theorem.  The following proposition implies the complex analog of \cite[Theorem 1.1]{Y} except that we also assume a second derivative bound.  Let ${\rm Herm}(n)$ denote the space of $n\times n$ Hermitian matrices.

\begin{lem}\label{lem: Liouville}
Suppose $u:\mathbb{C}^{n} \rightarrow \mathbb{R}$ is a $C^{3}$ function satisfying
\[
F(\de\dbar u) = \sigma.
\]
where $F :{\rm Herm}(n) \rightarrow \mathbb{R}$ is smooth and elliptic.  Assume that the set 
\[
\Gamma^{\sigma}=\{ M \in {\rm Herm}(n): F(M) >\sigma\}
\]
is convex.  If $|\de\dbar u|_{L^{\infty}(\mathbb{C}^{n})} \leq K < +\infty$, then $u$ is a quadratic polynomial.
\end{lem}

The proof follows by combining the convexity of the level sets of the equation $F(\de\dbar u) =\sigma$ with an extension trick in order to apply the standard Evans-Krylov estimate.  The extension trick occurs in two steps.  First we find a concave elliptic operator $F_{0}(\cdot)$, such that $F_{0}(\de\dbar u)=0$ if and only if $F(\de\dbar u) =\sigma$.  Secondly, we use a trick due to Wang \cite{YW}, which was used also by Tosatti-Wang-Weinkove-Yang \cite{TWWY}, to extend $F_0$ to a {\rm real} uniformly elliptic concave operator, to which we apply the Evans-Krylov theory.  While we expect this is well-known to experts, we give the details for the readers' convenience.

\begin{proof}
Let ${\rm Sym}(2n)$ denote the space of real symmetric $2n \times 2n$ matrices.  Note that we have a canonical inclusion $\iota: {\rm Herm}(n) \hookrightarrow {\rm Sym}(2n)$, and so we will always regard ${\rm Herm}(n) \subset {\rm Sym}(2n)$.  Let $\mathcal{H}_{\la,\Lambda} \subset {\rm Sym}(2n)$ denote the set of symmetric matrices with eigenvalues lying in $[\la, \Lambda]$.  

As in \cite{Szek},  we define $F_0: {\rm Herm}(n) \rightarrow \mathbb{R}$ by
\begin{equation}
F_{0}(A) := \inf \left\{ t : \la(A) - t\cdot Id \in \overline{\Gamma}^{\sigma} \right\},
\end{equation}
where $\la(A)$ denotes the eigenvalues of $A$.  The reader can check that $F_0$ is a smooth, elliptic, non-linear operator on ${\rm Herm}(n)$.  The convexity of  $\Gamma^{\sigma}$ implies that $F_0(\cdot)$ is a concave operator. Furthermore, $F_0(\de\dbar u)=0$ if and only if $F(\de\dbar u) =\sigma$.  Consider the compact, convex set
\[
B_{2K} := \left\{ M \in {\rm Herm}(n) : \| M\| \leq 2K \right\}.
\] 
Since $F_0(\cdot)$ is smooth, and elliptic, and $B_{2K}$ is compact, $F_0(\cdot)$ is uniformly elliptic on $B_{2K}$. 

The next step is to extend $F_0$ to a uniformly elliptic, concave operator outside of $B_{2K}$.  We use an envelope trick due to Wang \cite{YW} (see also \cite{TWWY}). The complex structure $J$ on $\mathbb{C}^{n}$ gives a canonical projection $p: {\rm Sym}(2n) \rightarrow {\rm Herm}(n)$, by setting
\[
p(M) = \frac{M + J^{T}MJ}{2}.
\]
Define
\[
\mathcal{B}_{2K} := \left\{ N \in {\rm Sym}(2n) : p(N) \in B_{2K} \right\},
\]
and extend $F_0$ to a smooth, concave operator $\hat{F}_{0} :\mathcal{B}_{2K}\rightarrow \mathbb{R}$ by setting
\[
\hat{F}_0(N) := F_0(p(N)).
\]
We claim that $\hat{F}_{0}$ is uniformly elliptic on $\mathcal{B}_{2K}$.  This is just a matter of linear algebra.  First, observe that if $M \geq0$ is positive semi-definite, then so is $p(M)$, since, for any vector $v\in \mathbb{R}^{2n}$,
\[
\langle v, p(M)v \rangle = \frac{ \langle v, Mv\rangle + \langle Jv, MJv\rangle}{2}.
\]
Furthermore, we clearly have $\Tr(p(M)) = \Tr(M)$.  From these two facts the uniform ellipticity of $\hat{F}_0$ on $\mathcal{B}_{2K}$ easily follows from the uniform ellipticity of $F_0$ on $B_{2K}$.  Hence, there are constants $0< \la < \Lambda <+\infty$ such that, for all $A \in \mathcal{B}_{2K}$ the differential of $F_0$, denoted $DF_0$, at $A$ lies in $\mathcal{H}_{\la,\Lambda}$.  We define
\begin{equation}
\begin{aligned}
F_1(N) := \inf \bigg\{L(N) :\quad &L: {\rm Sym}(2n) \rightarrow \mathbb{R} \text{ affine linear },\\
&DL \in \mathcal{H}_{\la, \Lambda}, \text{ and } L(A) \geq \hat{F}_0(A),\,\,\, \forall A \in \mathcal{B}_{2K}\bigg\}
\end{aligned}
\end{equation}
where $DL$ denotes the differential of $L$.  In words,  $F_1$ is the concave envelope of the graph of $\hat{F}_{0}$ over $\mathcal{B}_{2K}$.  As in \cite[Lemma 4.1]{TWWY}  it is straightforward to check that $F_1: {\rm Sym}(2n) \rightarrow \mathbb{R}$ is uniformly elliptic, concave and agrees with $\hat{F}_{0}$ over $\mathcal{B}_{2K}$.  Since $\de\dbar u: \mathbb{C}^{n} \rightarrow \mathcal{B}_{2K}$ we have
\[
F_1(D^2 u) =0.
\]
By the Evans-Krylov theorem \cite{E,K}, \cite[Theorem 6.1]{CC} and a standard scaling argument we have; for some $\beta = \beta (n,\la, \Lambda) \in(0,1)$ and for every $R>0$ there holds
\[
|D^2u|_{C^{\beta}(B_R(0))} \leq C(n,\la,\Lambda)R^{-\beta} \|D^2u\|_{L^{\infty}(B_{2R}(0))} \leq C(n,\la, \Lambda)R^{-\beta}K.
\]
Letting $R \rightarrow +\infty$ we get the result.
\end{proof}

We use this Liouville type result to conclude $C^{2,\beta}$ estimates by a blow-up argument.

\begin{lem}\label{lem: EK est}
Suppose $u: B_2 \subset \mathbb{C}^{n} \rightarrow \mathbb{R}$ is a smooth function satisfying 
\[
F(x,\de\dbar u) =h(x),
\]
for some smooth map $F: B_2 \times {\rm Herm}(n) \rightarrow \mathbb{R}$.  Suppose that $F(x,\cdot)$ is uniformly elliptic on $B_2 \times \de\dbar u(B_{2})$ with ellipticity constant $0<\la<\Lambda <+\infty$.  Assume $h:B_{2} \rightarrow [a,b]$ is $C^{2}$ and, for every $\sigma \in[a,b]$ and $x\in B_2$ the set $\Gamma^{\sigma} :=\{ M \in {\rm Herm}(n) : F(x,M)>\sigma \}$ is convex.  Then, for every $\beta \in(0,1)$ we have the estimate
\[
|\de \dbar u|_{C^{\beta}(B_{1/2})} \leq C(n, \beta, \la, \Lambda, |\de\dbar u|_{L^{\infty}(B_{2})}, \|h\|_{C^{2}(B_{2})}).
\]
\end{lem}

\begin{proof}
The proof is by a standard blow-up argument; see, for instance \cite{Co}.  We give the details for the convenience of the reader.  For each $x \in B_1$ consider the quantity
\[
N_u := \sup_{B_1}d_x |\de \de\dbar u|(x)
\]
where $d_x := {\rm dist}(x,\del B_1)$.  Suppose the supremum is achieved at $x_0 \in B_1$.  Consider the function $\ti u: B_{N_u}(0) \rightarrow R$ defined by
\[
\ti u(z) := \frac{N_u^2}{d_{x_0}^2}u\left(x_0+ \frac{d_{x_0}}{N_u}z\right) - A - A_iz_i
\]
where $A, A_i$ are chosen so that $\ti u(0) = \de \ti u(0) = 0$.  Note that
\[
\de\dbar \ti u = \de\dbar u, \qquad \|\de\de\dbar u\|_{L^{\infty}(B_{N_u}(0))} = |\de\de\dbar u(0)| =1.
\]
In particular, we have $|\de\dbar u|_{C^{\beta}(B_1)} \leq 1$ for every $\beta \in (0,1)$ and $\ti u$ solves
\[
F(x_0+ \frac{d_{x_0}}{N_u}z, \de\dbar \ti u(z)) = h\left(x_0+ \frac{d_{x_0}}{N_u}z\right), \qquad z\in B_{N_{u}}(0).
\]
Differentiating the equation in the $\de_\ell$ direction yields
\[
F^{i\bar{j}}(x_0+ \frac{d_{x_0}}{N_u}z,\de\dbar \ti u) \de_i\de_{\bar{j}} \de_\ell \ti u = \frac{d_{x_0}}{N_u}h'\left(x_0+ \frac{d_{x_0}}{N_u}z\right). 
\]
Since $F(x,\cdot)$ is uniformly elliptic and $h$ is smooth, the Schauder theory implies $\de \ti u$ is bounded in $C^{2,\beta}(B_{N_{u}/2}(0))$, and so $\ti u$ is controlled in $C^{3,\beta}(B_{N_{u}/2}(0))$.

Now, for the sake of finding a contradiction, suppose we have:
\begin{itemize}
\item a sequence $u_{n}: B_{2} \rightarrow \mathbb{R}$ such that $\|\de\dbar u_{n}\|_{L^{\infty}(B_2)} \leq K$, but so that $N_{u_{n}} \geq n$
\item functions $h_{n}:B_{2} \rightarrow [a,b]$ such that $\|h_n\|_{C^{2}(B_2)} \leq K$
\end{itemize}
For each $n$ let $x_{n} \in B_{1}$ be a point where $N_{u_{n}}$ is achieved.  By compactness, after passing to a subsequence (not relabelled) we may assume that:
\begin{itemize}
 \item $x_n \rightarrow x_{\infty} \in \overline{B_{1}}$.
 \item $h_{n}$ converges to some function $h$ uniformly in $C^{1,\beta'}$ topology on $B_{3/2}$ for some fixed $\beta' \in (0,1)$.
 \end{itemize}
 By the above rescaling we find functions $\ti u_{n} :B_{N_{u_{n}}}(0) \rightarrow \mathbb{R}$ such that 
 \begin{itemize}
 \item $\|\ti u_{n}\|_{C^{3,\beta}(B_{N_{u_{n}}}(0)) } \leq C$ and
\item $
F(x_n + \frac{d_{x_n}}{N_{u_{n}}}z,\de\dbar \ti u_{n}) =h_n\left(x_n + \frac{d_{x_n}}{N_{u_{n}}}z\right) \qquad z \in B_{N_{u_{n}}}(0).
$
\end{itemize}
Since $N_{u_{n}} \geq n$, a diagonal argument yields the existence of a function $u: \mathbb{C}^{n} \rightarrow \mathbb{R}$ and a subsequence (again, not relabelled) such that $\{u_{n}\}_{n\geq k}$ converges uniformly to $u$ in $C^{3,\alpha/2}(B_{k}(0))$.  In particular, we have
\[
F(x_0,\de\dbar u) = h(x_0), \qquad |\de\de\dbar u|(0)=1.
\]  
Clearly $h(x_0) \in [a,b]$, and so we may apply Lemma~\ref{lem: Liouville} to conclude that $u$ is a quadratic polynomial, which is a contradiction.  
\end{proof}

By arguing locally, Lemma~\ref{lem: EK est} immediately implies the following corollary, whose proof we leave to the reader, and finishes the proof of Theorem~\ref{thm: est thm}.

\begin{cor}
Suppose $u: X\rightarrow \mathbb{R}$ is a solution of
\[
\Theta_{\alpha}(\omega+\ddb u) = h(x)
\]
where $h(x)\geq (n-2)\frac{\pi}{2}+\epsilon$ for some $\epsilon>0$.  Then for every $\beta \in (0,1)$ we have the estimate
\[
|\de \dbar u|_{C^{\beta}(X)} \leq C(n, X,\alpha, \beta, \|h\|_{C^2}, \|\de\dbar u\|_{L^{\infty}(X)})
\]
\end{cor}

\section{The Method of Continuity and the proof of Theorem~\ref{thm: existence thm}}

In this section we prove Theorem~\ref{thm: existence thm}, using the method of continuity.  Unfortunately, the naive method of continuity does not work due essentially to the fact that the subsolution condition is non-trivial; for related discussion see \cite{Szek}.  Instead, adapting an idea of Sun \cite{WS} in the setting of the $J$-equation, the proof of Theorem~\ref{thm: existence thm} requires two applications of the method of continuity.  Let us first prove openness along a general method of continuity. 
\begin{lem}\label{lem: openness}
Fix $k \geq 2, \beta \in (0,1)$ and suppose we have $C^{k-2,\beta}$ functions $H_0, H_1 :X \rightarrow \mathbb{R}$, and a $C^{k,\beta}$ function $u: X\rightarrow \mathbb{R}$ such that
\[
\Theta_{\alpha}(\omega +\ddb u) = H_0.
\]
Consider the family of equations
\begin{equation}\label{eq: general MOC}
\Theta_{\alpha}(\omega +\ddb u_t) = (1-t)H_0 + tH_1+c_t
\end{equation}
for $c_t$ a constant.  There exists $\epsilon>0$ such that, for every $|t|<\epsilon$ a unique pair $(u_t,c_t) \in C^{k,\beta} \times \mathbb{R}$ solving \eqref{eq: general MOC}.  Furthermore, if $H_0, H_1$ are smooth, then so is $u_t$. 
\end{lem}
\begin{proof}
The proof is by the implicit function theorem. Fix $\beta>0$, $k \geq 2$ and consider the map $F : [0,1] \times C^{k,\beta}\times \mathbb{R} \rightarrow C^{k-2,\beta}$ given by
\[
(t,c,u) \longmapsto \Theta_{\alpha}(\omega +\ddb u) -(1-t)H_0 -tH_1-c.
\]
Let $\Delta_\eta$ denote the linearization of $\Theta_{\alpha}$ around $(u_0, c_0) := (u,0)$.  The operator $\Delta_{\eta}$ is homotopic to the Laplacian with respect to $\alpha$, and so has index $0$.  By the maximum principle, the kernel of $\Delta_{\eta}$ consists of the constants, and hence the cokernel of $\Delta_{\eta}$ has dimension $1$.  Another application of the maximum principle shows that the constants are not in the image of $\Delta_{\eta}$.  It follows that the linearization of $F$ at time $0$, given by
\[
(v,c) \longmapsto \Delta_{\eta}v +c,
\]
is a surjective map from $C^{k,\beta} \times \mathbb{R}$ to $C^{k-2,\beta}$.  In particular, by the implicit function theorem we conclude that there exists $\epsilon>0$ such that, for all $|t|<\epsilon$ we can find a unique pair $(u_t,c_t) \in C^{k,\beta}\times \mathbb{R}$ solving~\eqref{eq: general MOC}.  By a standard boot strapping argument, we find that $u_{t}$ is in fact smooth provided $H_0, H_1$ are smooth. 
\end{proof}

Suppose now that we have a subsolution $\chi \in [\Omega]$ to the deformed Hermitian-Yang-Mills equation satisfying the assumptions of Theorem~\ref{thm: existence thm}.  Let us denote by
\[
\Theta_0 := \Theta_{\alpha}(\chi).
\]
Without loss of generality, we will assume that $\Theta_0 \ne \hat{\Theta}$, for otherwise we are finished.  Now, and for the remainder of this section, we let $\mu_1,\cdots \mu_n$ be the eigenvalues of the relative endomorphism $\alpha^{-1}\chi$ at an arbitrary point of $X$.  We clearly have
\[
\sum_{i\ne j} \arctan(\mu_i) > \Theta_0 -\frac{\pi}{2} \qquad \forall j.
\]
In particular, we can find $\delta_0 >0$ such that
\[
\sum_{i\ne j} \arctan(\mu_i) > \max\{\Theta_0, \hat{\Theta}\} + 100\delta_0 -\frac{\pi}{2} \qquad \forall j.
\]
Furthermore, since 
\[
{\rm Arg} \int_{X} (\alpha+ \sqrt{-1}\chi)^{n} = \hat{\Theta}
\]
we must have that $\inf_X\Theta_0 < \hat{\Theta}$.  Choose $\delta_1>0$ such that
\[
\inf_X \Theta_0 + 100 \delta_1 = \hat{\Theta}
\]
Set $\delta = \min\{ \delta_0, \delta_1\}$, and define
\[
\Theta_1 = \widetilde{\max}_{\delta} \{\hat{\Theta}, \Theta_0\}
\]
where $\widetilde{\max}_{\delta}$ denotes the regularized maximum \cite{DemB}.  We have
\begin{lem}\label{lem: key props MOC}
Fix a point $p \in X$ where $\Theta_0$ achieves its infimum.  The function $\Theta_1$ has the following properties:
\begin{enumerate}
\item[{\it (i)}]  $\Theta_1$ is smooth.
\item[{\it (ii)}]  $\max\{\Theta_0, \hat{\Theta}\} \leq \Theta_1 \leq  \max\{\Theta_0, \hat{\Theta}\}+\delta$.
\item[{\it (iii)}]  $\Theta_1(x) = \hat{\Theta}$ on the set $\{x \in X : \Theta_0+\delta \leq \hat{\Theta}-\delta\}$.  In particular, $\Theta_1(x) = \hat{\Theta}$ in a neighbourhood of $p \in X$.
\item[{\it (iv)}]   $\Theta_1(x) = \Theta_0(x)$ on the set $\{x \in X : \hat{\Theta}+\delta \leq \Theta_0-\delta\}$.
\item[{\it (v)}] For every $t \in [0,1]$ 
\[
\inf_{X}[(1-t)\Theta_0 +t\Theta_1] = (1-t)\inf_{X} \Theta_0 +t \hat{\Theta} = (1-t)\Theta_0(p) + t\hat{\Theta}.
\]
\item[{\it (vi)}]  $\sup_{X} [\Theta_1 -\Theta_0]= \Theta_1(p) -\Theta_0(p) = \hat{\Theta} - \inf_X\Theta_0(p)$.
\end{enumerate}
\end{lem} 
\begin{proof}
Statements ${\it (i)-(iv)}$ are just the properties of the regularized maximum, \cite[Chapter 1, Lemma 5.18]{DemB}.  We prove ${\it (v)}$.  From our choice of $\delta$, and the definition of $\Theta_1$ we have $\Theta_1(p) = \hat{\Theta}$.  Thus
\[
\begin{aligned}
(1-t)\Theta_0(p) + t\hat{\Theta} &= (1-t)\Theta_0(p) + t\Theta_1(p) \\
&\geq \inf_{X} [(1-t)\Theta_0 + t\Theta_1] \\
&\geq (1-t)\inf_X\Theta_0 +t\inf_X \Theta_1 \\
&= (1-t)\Theta_0(p) +t\hat{\Theta},
\end{aligned}
\]
establishing the fifth point.  For ${\it (vi)}$, we first consider the set $U_1 := \{x \in X : \Theta_0+\delta \leq \hat{\Theta}-\delta\}$.  On this set we have $\Theta_1 -\Theta_0 = \hat{\Theta} -\Theta_0$ by property ${\it (iii)}$.  This difference is maximized at the point $p \in U_1$, where we have
\[
\Theta_1(p) -\Theta_0(p) = \hat{\Theta} -\Theta_0(p) = 100\delta_1 \geq 100\delta.
\]
Now consider the set $U_2 := \{x \in X : \hat{\Theta}+\delta \leq \Theta_0-\delta\}$.  On this set we have $\Theta_1 - \Theta_0 \equiv 0$ by ${\it (iv)}$.  Finally, we consider the set $U_{3} = \{x \in X : |\Theta_0 -\hat{\Theta}|<2\delta\}$.  On $U_3$ we have
\[
\Theta_1 -\Theta_0 \leq \max\{\Theta_0 ,\Hat{\Theta}\} +\delta -\Theta_0 \leq 3\delta < 100\delta,
\]
and the lemma follows.
\end{proof}

We use the function $\Theta_1$ as the first target for the method of continuity. 

\begin{prop}\label{prop: MOC1}
There exists a smooth function $u_1:X \rightarrow \mathbb{R}$, and a constant $b_1<0$ such that
\[
\Theta_{\alpha}(\omega +\ddb u_1) = \Theta_1 +b_1, \qquad \text{ and} \qquad \Theta_1 +b_1 > (n-2)\frac{\pi}{2}.
\]
\end{prop}
\begin{proof}
We use the method of continuity.  Consider the family of equations 
\begin{equation}\label{eq: first MOC}
\Theta_{\alpha}(\chi+\ddb u_t) = (1-t)\Theta_0 +t\Theta_1 +b_t.
\end{equation}
Define 
\[ 
I = \big\{t \in[0,1]: \exists\,\, (u_{t},b_t) \in C^{\infty}(X) \times \mathbb{R} \text{ solving } \eqref{eq: first MOC} \big\}.
\]
 Since $(0,0)$ is a solution at time $t=0$, we have that $I$ is non-empty.  By Lemma~\ref{lem: openness} the set $I$ is open.  It suffices to prove that $I$ is closed.  This will follow from the a priori estimates in Theorem~\ref{thm: est thm} together with a standard bootstrapping argument provided we can show
\begin{itemize}
\item $\chi$ is a subsolution of equation~\eqref{eq: first MOC} for all $t\in[0,1]$
\item $(1-t)\Theta_0 +t\Theta_1 +b_t > (n-2)\frac{\pi}{2}$ uniformly for $t\in[0,1]$.
\end{itemize}
In order to do each of these things, we must control the constant $b_t$.  First, we claim that $b_{t} \leq t\sup_{X}(\Theta_0-\Theta_1) \leq 0$.  To see this, choose $q\in X$ where $u_t$ achieves its maximum.  Then at $q$, ellipticity implies
\[
\Theta_{0}(q) \geq \Theta_{\alpha}(\chi +\ddb u_t)(q) = (1-t)\Theta_0(q) +t\Theta_1(q) +b_t.
\]
Rearranging this equation yields
\[
b_{t} \leq t\sup_{X}(\Theta_0-\Theta_1) \leq 0
\]
where the final inequality follows from the fact that $\Theta_1 \geq \Theta_0$ by construction.  It follows that for every $1\leq j\leq n$ there holds,
\begin{equation*}
\begin{aligned}
\sum_{i\ne j} \arctan(\mu_i) &> \max\{\Theta_0, \hat{\Theta}\} + 100\delta - \frac{\pi}{2}\\
&> \Theta_1 -\frac{\pi}{2}\\
& \geq (1-t)\Theta_0 +t\Theta_1 +b_t  -\frac{\pi}{2}
\end{aligned}
\end{equation*}
and so $\chi$ is a subsolution of equation~\eqref{eq: first MOC} for all $t\in [0,1]$, taking care of the first point.  To take care of the second point we look at a point $q \in X$ where $u_t$ achieves its minimum to find
\[
b_t \geq -t\sup_{X}(\Theta_1-\Theta_0).
\]
Combining this estimate with the results of Lemma~\ref{lem: key props MOC}, we have
\begin{equation}
\begin{aligned}
\inf_X\left[ (1-t)\Theta_0 +t\Theta_1 +b_t\right]  &= (1-t)\Theta_0(p) +t\Theta_1(p) +b_{t} \\
&= \Theta_0(p) +t(\Theta_1(p) -\Theta_0(p)) +b_t\\
&= \Theta_0(p) +t\sup_{X} (\Theta_1-\Theta_0) +b_t\\
&\geq \Theta_0(p) \\
&> (n-2)\frac{\pi}{2}.
\end{aligned}
\end{equation}
By Theorem~\ref{thm: est thm}, together with the usual Schauder estimates and bootstrapping argument we conclude that $I$ is closed. Proposition~\ref{prop: MOC1} follows.
\end{proof}

We now turn to the proof of the main theorem.  Let $\omega_1 = \chi +\ddb u_1$, where $u_1$ is the function from Proposition~\ref{prop: MOC1}.   We consider the method of continuity
\begin{equation}\label{eq: MOC 2}
\Theta_{\alpha}(\omega_1 +\ddb v_t) = (1-t)\Theta_1 + t\hat{\Theta} +c_t.
\end{equation}
Define\[
I = \{t \in[0,1]: \exists\,\, (v_{t},c_t) \in C^{\infty}(X) \times \mathbb{R} \text{ solving } \eqref{eq: MOC 2} \}.
\]
By Proposition~\ref{prop: MOC1} we have a solution a time $t=0$, with constant $c_0 = b_1$.  Thanks to Lemma~\ref{lem: openness}, the set $I$ is open, and so it suffices to prove $I$ is closed.  Again this will follow from the a priori estimates in Theorem~\ref{thm: est thm}, if we can show that
\begin{itemize}
\item $\chi$ is a subsolution along the whole method of continuity \eqref{eq: MOC 2}.
\item  $(1-t)\Theta_1 + t\hat{\Theta} +c_t > (n-2)\frac{\pi}{2}$ for all $t\in [0,1]$.
\end{itemize}
as in the proof of Proposition~\ref{prop: MOC1}, it suffices to control the constant $c_t$.  To control $c_t$ from above we observe that since $\Theta_1 \geq \hat{\Theta}$, the cohomological condition
\[
{\rm Arg}\int_{X} \sqrt{\frac{\det \eta_t}{\det \alpha}} e^{i\left((1-t)\Theta_1 + t\hat{\Theta} +c_t\right)} \alpha^{n} = \hat{\Theta}
\]
implies that $c_t \leq 0$ for all $t \in [0,1]$.  Arguing as in the proof of Proposition~\ref{prop: MOC1}, we conclude that $\chi$ is again a subsolution along the whole method of continuity.  Furthermore, if $p\in X$ is a point where $\Theta_0$ achieves its infimum, then Lemma~\ref{lem: key props MOC} part ${\it (iii)}$, combined with Proposition~\ref{prop: MOC1} implies
\[
\hat{\Theta}+b_{1} = \Theta_1(p) +b_1 > (n-2)\frac{\pi}{2},
\]
and so in particular, we have
\[
(1-t)[\Theta_1+b_1] + t[\hat{\Theta}+b_1] > (n-2)\frac{\pi}{2}.
\]
In order to show that $(1-t)\Theta_1 + t\hat{\Theta} +c_t > (n-2)\frac{\pi}{2}$ it suffices to show that $c_{t} \geq b_{1}$ for all $t$.  This is easy.  If the maximum of $v_{t}$ is achieved at the point $q\in X$, then we have
\[
\Theta_1(q) +b_{1} \leq (1-t)\Theta_1(q) + t\hat{\Theta} +c_t.
\]
or in other words,
\[
c_t \geq b_1 +t[\Theta_1(q) -\hat{\Theta}] \geq b_1
\]
since $\Theta_1 \geq \hat{\Theta}$.  As a result we can apply the a priori estimates in \ref{thm: est thm} uniformly in $t$ to conclude that $I$ is closed.  The higher regularity follows in the usual way from the Schauder estimates and bootstrapping.  Since we clearly have $c_1=0$ by the cohomological condition, Theorem~\ref{thm: existence thm} follows.

\begin{rk}
It is easy to establish the following weaker existence theorem using the parabolic flow introduced in \cite{JY}:  If $\hat{\Theta} >(n-1)\frac{\pi}{2}$, and $\chi \in \Omega$ is a subsolution with $\Theta_{\alpha}(\chi) > (n-1)\frac{\pi}{2}$, then the flow in \cite{JY} starting at $\chi$ converges smoothly to a solution of the deformed Hermitian-Yang-Mills equation. 
\end{rk}

\section{Subsolutions, Class conditions and Stability}\label{sec: stability}

In this section we briefly elaborate on the subsolution condition as well as pose some natural conjectures related to the existence of solutions to the deformed Hermitian-Yang-Mills equation.  The first step is to observe that the subsolution condition in Lemma~\ref{lem: sub sol def} is equivalent to a class condition, as we alluded to in Remark~\ref{rk: form type pos}.  Recall that $\Omega\in H^{1,1}(X,\mathbb{R})$ is a fixed cohomology class. We then have the following proposition.

\begin{prop}\label{prop: form type pos}
Let $\hat{\Theta} \in ((n-2)\frac{\pi}{2}, n\frac{\pi}{2})$ be the fixed constant defined in Section \ref{background}.  Then a $(1,1)$ form $\chi \in\Omega$ is a subsolution to equation~\eqref{eq: SLag type} if and only if
\begin{enumerate}
\item $\dim _{\mathbb{C}} X=n$ is even and
\begin{equation}\label{eq: sub sol form}
 -(i^n)\left( {\rm Im}(\alpha+ \sqrt{-1}\chi)^{n-1}+\cot(\hat{\Theta}) {\rm Re}(\alpha+\sqrt{-1}\chi)^{n-1}\right)>0
\end{equation}
\item $\dim _{\mathbb{C}} X=n$ is odd and
\[
i^{n-1}\left(\tan(\hat{\Theta}) {\rm Im}(\alpha+ \sqrt{-1}\chi)^{n-1}+{\rm Re}(\alpha+\sqrt{-1}\chi)^{n-1}\right)>0.
\]
\end{enumerate}
In each line, positivity is to be understood in the sense of $(n-1,n-1)$ forms.

\end{prop}

\begin{proof}
We will prove the statement in the case $n\equiv0$ (mod $4)$, as all other cases are similar.  Suppose that $\chi$ is a subsolution in the sense of Lemma~\ref{lem: sub sol def}.   Since the statement is pointwise, it suffices to  fix a point $x_{0} \in X$, and coordinates so that $\alpha$ is the identity at $x_{0}$ and $\chi (x_{0})$ is diagonal with entries $\mu_{1}, \dots, \mu_{n}$.  By assumption, for every $1\leq j\leq n$ we have
\[
(n-1)\frac{\pi}{2} >\sum_{i\ne j} \arctan(\mu_{i}) > \hat{\Theta}-\frac{\pi}{2}.
\]
In other words
\begin{equation}\label{eq: Arg subsol}
(n-1)\frac{\pi}{2}>{\rm Arg}\left(\prod_{j\ne i} (1+\sqrt{-1} \mu_{j})\right) > \hat{\Theta}-\frac{\pi}{2},
\end{equation}
where again we have fixed the branch cut of ${\rm Arg}$ by setting it to be zero when $\mu_{1} = \cdots = \mu_n =0$.  If $\hat{\Theta} = (n-1)\frac{\pi}{2}$, then the fact that $n\equiv0$ (mod $4)$ along with ~\eqref{eq: Arg subsol} implies
\[
{\rm Im}\left(\prod_{j\ne i} (1+\sqrt{-1} \mu_{j})\right)<0.
\]
Since $\cot(\hat{\Theta})=0$ in this case, we obtain~\eqref{eq: sub sol form}.  Otherwise,~\eqref{eq: Arg subsol} implies
\[
\arctan\left( \frac{ {\rm Im} \prod_{i\ne j} (1+\sqrt{-1}\mu_{i})}{{\rm Re}\prod_{i\ne j} (1+\sqrt{-1}\mu_{i})}\right) > \hat{\Theta} -\frac{\pi}{2}- k\pi
\]
where on the right hand side, we choose $k\in \mathbb{Z}$ so that $\hat{\Theta} -\frac{\pi}{2}- k\pi \in (-\frac\pi2,0)\cup(0, \frac\pi2)$.  Since $\tan(\cdot)$ is increasing and non-zero on $(-\frac\pi2,\frac\pi2)$, we obtain
\[
\frac{ {\rm Im} \prod_{i\ne j} (1+\sqrt{-1}\mu_{i})}{{\rm Re}\prod_{i\ne j} (1+\sqrt{-1}\mu_{i})} > -\cot( \hat{\Theta}).
\]
Above we have used the elementary fact that $\tan(x-\pi/2) = -\cot(x)$ for $x\ne0$ (mod $\pi)$.  By~\eqref{eq: Arg subsol}, the complex number $\prod_{j\ne i} (1+\sqrt{-1} \mu_{j})$ has argument lying in the interval $((n-3)\frac{\pi}{2}, (n-1)\frac{\pi}{2})$.  Since $n\equiv0$ (mod $4)$, this implies that it has negative real part. As a result, we have
\[
 {\rm Im} \prod_{i\ne j} (1+\sqrt{-1}\mu_{i}) <-\cot( \hat{\Theta}){\rm Re}\prod_{i\ne j} (1+\sqrt{-1}\mu_{i}). 
\]
Since this holds for all $j$, we obtain that~\eqref{eq: sub sol form} holds in the sense of $(n-1,n-1)$ forms.  

The reverse implication holds by essentially the same argument.  Suppose that $\chi$ satisfies~\eqref{eq: sub sol form}.  Since $\chi\in\Omega$ we get
\[
{\rm Arg} \int_{X}  (\alpha +\sqrt{-1} \chi)^{n} = \hat{\Theta}.
\]
It follows that there exists a point $x_{0} \in X$ such that $\Theta_{\alpha}(\chi)= \hat{\Theta}$.  In particular, in a neighbourhood of $x_{0}$,  $\chi$ defines a subsolution in the sense of Lemma~\ref{lem: sub sol def}.  The set of points $U \subset X$ where $\chi$ defines a subsolution is thus open and non-empty.  On the other hand, it is also closed.  To see this assume we can find points $p_{j} \in U$ converging to $p$, and at $p$ there exists a $j$ such that
\[
\sum_{i\ne j} \arctan(\mu_{j}) = \hat{\Theta}- \frac{\pi}{2}.
\]
The above computation implies that, at $p$ the $(n-1,n-1)$ form
\[
- {\rm Im}(\alpha+ \sqrt{-1}\chi)^{n-1}-\cot(\hat{\theta}) {\rm Re}(\alpha+\sqrt{-1}\chi)^{n-1}
\]
is positive, but {\em not} strictly positive, which is a contradiction.  Since $X$ is connected, it follows that $\chi$ is a subsolution everywhere.
\end{proof}

Notice that if $\chi$ is a subsolution to the deformed Hermitian-Yang-Mills equation \eqref{eq: SLag type} in the sense of Lemma~\ref{lem: sub sol def}, then in fact we obtain a large set inequalities that the eigenvalues of $\chi$ with respect to $\alpha$ must satisfy.  Namely, at a point $x_{0} \in X$, and in coordinates so that $\alpha$ is the identity at $x_{0}$ and $\chi(x_{0})$ is diagonal with entries $\mu_{1}, \dots, \mu_{n}$, then for every choice of $\ell$ distinct indices $j_{1}, \dots, j_{\ell}$, and every $1\leq \ell \leq n-1$, we must have
\[
\sum_{i\notin\{j_1,\dots,j_n\}} \arctan(\mu_{i}) \geq \hat{\Theta} - \ell \frac{\pi}{2}.
\]
Of course, for any $\ell >1$, these inequalities are all implied by the definition of a subsolution, so we have not really gained anything new.  On the other hand, this observation suggests a cohomological obstruction to the existence of solutions for the deformed Hermitian-Yang-Mills equation.  In order to explain this, we first prove

\begin{lem}\label{lem: arg sub var}
A $(1,1)$ form $\chi \in [\omega]$ is a subsolution to the deformed Hermitian-Yang-Mills equation if and only if, for any $1\leq p \leq n-1$, and any non-zero, simple, positive $(n-p,n-p)$ form $\beta$, we have
\begin{equation}
{\rm Arg}\frac{ (\alpha+\sqrt{-1}\chi)^{p}\wedge \beta}{\alpha^{n}} > \hat{\Theta} - (n-p)\frac{\pi}{2}.
\end{equation}
\end{lem}
\begin{proof}
The proof is a matter of linear algebra.  Recall that a smooth $(k,k)$ form $\beta$ defined on an open set is said to be a simple positive form if it can be written as
\[
\beta = (\sqrt{-1})^{k} \beta_1\wedge\overline{\beta_{1}} \wedge \beta_2\wedge\overline{\beta_{2}}\wedge \dots \wedge\beta_k\wedge\overline{\beta_{k}}
\]
for smooth $(1,0)$ forms $\beta_{j}$ \cite{Ko}.  Since the statement is pointwise, we again fix a point $x_{0} \in X$, and coordinates so that $\alpha$ is the identity at $x_{0}$ and $\chi(x_{0})$ is diagonal with entries $\mu_{1}, \dots, \mu_{n}$.  For any $p$, we can have
\[
(\alpha+\sqrt{-1}\chi)^{p} = (\sqrt{-1})^{p} p! \sum_{J} \prod_{j\in J} (1+\sqrt{-1} \mu_{j}) dz^{J} \wedge d\overline{z^{J}}
\]
where the sum is over ordered sets $J \subset \{1,\dots,n\}$ of cardinality $p$.  Suppose that $\chi$ is a subsolution. Then by the above remarks we know that, for any $J$ we have
\[
{\rm Arg}\prod_{j\in J} (1+\sqrt{-1} \mu_{j}) > \hat{\Theta}-(n-p)\frac{\pi}{2}.
\]
Let $\beta$ be any non-zero simple positive $(n-p,n-p)$ form.  Then we can write
\[
\beta = (\sqrt{-1})^{n-p} \sum_{J} c_{J} dz^{J^{c}}\wedge d\overline{z^{J^{c}}} + \tilde{\beta}
\]
for a smooth $(n-p,n-p)$ form $\tilde{\beta}$ satisfying $\tilde{\beta} \wedge dz^{J} \wedge d\overline{z^{J}}=0$ for all $J$.  Here again the sum is over ordered sets $J \subset \{1,\dots,n\}$ of cardinality $p$, and $J^{c}$ denotes the ordered complement of $J$.  The coefficients $c_{J}$ are necessarily real, non-negative, and at least one $c_{J}$ must be strictly positive since
\[
p!(\sqrt{-1})^n\sum_{J} c_{J} dz_{1}\wedge d\overline{z_1}\wedge\cdots\wedge dz_{n}\wedge d\overline{z_n} = \beta \wedge \alpha^p.
\]
The right hand side is positive and not identically zero, since $\beta$ is non-zero.  Thus we have
\[
 \frac{(\alpha+\sqrt{-1}\chi)^{p}\wedge \beta}{\alpha^{n}} = \sum_{J} c_{J}\left(\prod_{j\in J} (1+\sqrt{-1} \mu_{j})\right).
 \]
 The right hand side is a positive linear combination of complex numbers with arguments strictly larger than $\hat{\Theta}-(n-p)\frac{\pi}{2}$, and so
 \[
 {\rm Arg}\frac{ (\alpha+\sqrt{-1}\chi)^{p}\wedge \beta}{\alpha^{n}} > \hat{\Theta} - (n-p)\frac{\pi}{2}.
 \]
The reverse implication is trivial, by taking $\beta =(\sqrt{-1})^{n-p}dz^{J^{c}}\wedge d\overline{z^{J^{c}}}$ for every ordered set $J$ of cardinality $p$.

\end{proof}

The upshot of this linear algebra is the following proposition, which is essentially a corollary of Lemma~\ref{lem: arg sub var}
\begin{prop}\label{prop: stable nec}
For every subvariety $V \subset X$, define 
\begin{equation}
\Theta_{V} := {\rm Arg} \int_{V} (\alpha+\sqrt{-1}\omega)^{\dim V}.
\end{equation}
If there exists a solution to the deformed Hermitian-Yang-Mills equation, then for every proper subvariety $V\subset X$ we have
\[
\Theta_{V} > \Theta_{X} - (n-\dim V)\frac{\pi}{2}.
\]
\end{prop}

We can recast the condition in Proposition~\ref{prop: stable nec} in the following way.  For every subvariety $V \subset X$ define a complex number
\[
Z(V) := -\int_{X} e^{-\sqrt{-1}\alpha +\omega}
\]
where by convention we only integrate the term in the expansion of order $\dim V$.  Note that $Z(V)$ differs from the complex number $\int_V (\alpha+\sqrt{-1}\omega)^{\dim V}$ only by factors of $(\sqrt{-1})$.  When $[\omega] = c_{1}(L)$, this formula is equivalent to
\begin{equation}\label{eq: central charge}
Z(V) = -\int_{V}e^{-\sqrt{-1}\alpha}{\rm ch}(L).
\end{equation}
It is easy to check that if $\Theta_{X} \in ((n-2)\frac{\pi}{2}, n\frac{\pi}{2})$, then $Z(X)$ lies in the upper half plane.  Let us denote by ${\rm Arg}_{p.v.}$ the principal value of ${\rm Arg}$, valued in $(-\pi, \pi]$.  Then, in the notation of Proposition~\ref{prop: stable nec} we have $\Theta_V > \Theta_X -(n-\dim V) \frac{\pi}{2}$ implies
\[
{\rm Arg}_{p.v.}\,Z(V) > {\rm Arg}_{p.v.}\,Z(X).
\]
The reader can easily check that the converse also holds, provided we assume that $\Theta_V > (\dim V-2)\frac{\pi}{2}$.  The numbers $Z(V)$ appearing in~\eqref{eq: central charge} bear a resemblance to the various notions central charge appearing in stability conditions in several physical and mathematical theories.  For example, we refer the reader to the works of Douglas \cite{Doug, Doug1, Doug2}, Bridgeland \cite{Br}, and Thomas \cite{T} to name just a few.  We hope to further elucidate this observation in future work.

Additionally, the condition appearing in Proposition~\ref{prop: stable nec} is, at least heuristically, similar to the algebro-geometric stability notions appearing in other problems in complex geometry.  Perhaps most notably, the notion of Mumford-Takemoto stability pertaining to the existence of Hermitian-Einstein metrics on holomorphic vector bundles \cite{Don1, UY}, and the recent stability condition posed by Lejmi-Sz\'ekelyhidi for the convergence of the $J$-equation, and more generally existence of solutions to the inverse $\sigma_k$-equations \cite{LS}.  Let us briefly recount this conjecture in the setting of the $J$-equation.

\begin{conj}[\cite{LS}]\label{conj: J flow}
Let $(X,\alpha)$ be a K\"ahler manifold, and $[\omega]$ another K\"ahler class.  For every subvariety $V\subset X$ with $\dim V=p$ define
\[
c_{V} := \frac{p\int_{V}\omega^{\dim p-1}\wedge\alpha}{\int_{V}\omega^{p}}.
\]
Then there exists a solution to the $J$-equation if and only if $c_{X}>c_{V}$ for all proper subvarieties $V\subset X$.
\end{conj}

This conjecture is known to hold when $\dim X=2$, thanks to the third authors solution of the Calabi conjecture \cite{Yau} and work of Demailly-P\u aun \cite{DP}.  The conjecture also holds when $X$ is a complex torus, due to Lejmi-Sz\'ekelyhidi.  Recently, the first author and Sz\'ekelyhidi \cite{CS} have proven the conjecture in the case that $X$ is toric.  It is interesting to note that the stability condition in Conjecture~\ref{conj: J flow} arises from a modification of K-stability by considering certain special test configurations arising from deformation to the normal cone.  We expect that the stability type condition in Proposition~\ref{prop: stable nec} can be realized in a similar manner, a point which we will address in future work.  Finally, we note;

\begin{prop}\label{prop: conj dim 2}
If $\dim X=2$, then a solution to the deformed Hermitian-Yang-Mills equation exists if and only if, for every curve $C\subset X$ we have
\begin{equation}\label{eq: stab dim 2}
\Theta_{C} > \Theta_{X}-\frac{\pi}{2}.
\end{equation}
\end{prop}
\begin{proof}
Let us assume that $\Theta_{X} >0$.  If $\Theta_{X}<0$, then we can replace $[\omega]$ with $[-\omega]$, and if $\Theta_{X}=0$, then the condition in~\eqref{eq: stab dim 2} is vacuous, and a solution always exists, as observed in \cite{JY}.  We can there for assume that $\Theta_{X} \in(0,\pi)$, and so
\begin{equation}\label{eq: dim 2 norm}
1-\int_{X}\omega^{2} = 2\cot(\Theta_{X})\int_{X}\alpha\wedge \omega.
\end{equation}
It was observed in \cite{JY} that a solution to the deformed Hermitian-Yang-Mills equation exists if and only if the class $[\cot(\Theta_{X})\alpha+\omega]$ is K\"ahler, thanks to the third authors solution of the Calabi conjecture \cite{Yau}.  Since $[\alpha]$ is Kahler, the class $[\Omega_{T}]:=[(T+\cot(\Theta_{X}))\alpha+\omega]$ is a K\"ahler class for $T \gg0$.  Suppose there exists a time $T \geq0$ where  $[\Omega_{T}]$ lies on the boundary of the K\"ahler cone-- that is, $[\Omega_{T}]$ is nef, but not K\"ahler.  First, we claim that $[\Omega_{T}]$ is big.  By \cite[Theorem 2.12]{DP} it suffices to check that $\int_{X}\Omega_{T}^{2} >0$.  We compute
\begin{equation*}
\begin{aligned}
\int_{X} \Omega_{T}^{2} &= (T+\cot(\Theta_{X}))^{2} + 2(T+\cot(\Theta_{X}))\int_{X}\alpha\wedge\omega + \int_{X}\omega^2\\
&=(T+\cot(\Theta_{X}))^{2} +1+ 2T\int_{X}\alpha\wedge \omega,
\end{aligned}
\end{equation*}
where we have used~\eqref{eq: dim 2 norm}.  Note that since $\Theta_{X} \in(0, \pi)$ we have
\[
2\int_{X}\alpha \wedge \omega = {\rm Im}(\alpha+\sqrt{-1}\omega)^{2} >0,
\]
and so the above computation implies
\[
\int_{X}\Omega_{T}^{2} \geq 1.
\]
Finally, by the main theorem of \cite{DP} (see also \cite{CT}), we can conclude that $[\Omega_{T}]$ is K\"ahler provided $\int_{C}\Omega_{T} > 0$ for any curve $C\subset X$.  Fix $C \subset X$.  Since $\Theta_{C} > \Theta_{X}-\pi/2$, we know $\Theta_{C}\in(-\frac{\pi}{2}, \frac{\pi}{2})$, and so the following equality makes sense;
\[
\tan(\Theta_{C}) \int_{C} \alpha = \int_{C} \omega.
\]
Because $\tan(\cdot)$ is defined an increasing on $(-\frac{\pi}{2}, \frac{\pi}{2})$, we have 
\[
\tan(\Theta_{C}) > \tan(\Theta_{X}-\frac{\pi}{2}) = -\cot(\Theta_{X}).
\] 
Furthermore, $\Theta_{C} \in (-\frac{\pi}{2}, \frac{\pi}{2})$ implies that 
\[
\int_{C}\alpha = {\rm Re}\int_{C}\alpha+\sqrt{-1}\omega >0
\]
and so we obtain
\[
\int_{C}\omega = \tan(\Theta_{C}) \int_{C} \alpha  \geq -\cot(\Theta_{X})\int_{C}\alpha.
\]
Since $T\geq 0$,
\[
\int_{C}\Omega_{T} = T\int_{C} \alpha + \int_{C}\cot(\Theta_{X})\alpha+\omega > T\int_{C}\alpha >0,
\]
and so $[\Omega_{T}]$ is K\"ahler as long as $T\geq 0$, and the proposition follows.
\end{proof}

We end by remarking that one could hope for a similar framework for the lower branches of the deformed Hermitian-Yang-Mills equation-- that is, when $\Theta_{X} \leq (n-2)\frac{\pi}{2}$.  However, due to the lack of convexity in the lower branches we expect that the deformed Hermitian-Yang-Mills equation with subcritical phase may be extremely poorly behaved from an analytic and algebraic stand point.  For example, in the real case Nadirashvili-Vl\u adu\c{t} \cite{NV} and Wang-Yuan \cite{WY1} have demonstrated the existence of $C^{1,\beta}$ viscosity solutions to the special Lagrangian equation with subcritical phase on a ball in $\mathbb{R}^{3}$ for $n\geq 3$ which are not $C^2$ in the interior.  Furthermore, Wang-Yuan \cite{WY1} have shown the existence of smooth solutions $\{u^{\epsilon}\}$ to the special Lagrangian equation with fixed, subcritical phase on a ball in $\mathbb{R}^{3}$ such that $\|Du^{\epsilon}\|_{L^{\infty}} <C$, but so that $|D^2u^{\epsilon}|(0)$ blows up as $\epsilon \rightarrow 0$.

\bigskip


\begin{thebibliography}{99}
\bibitem{Bl} Z. B\l ocki, {\em On uniform estimate in Calabi-Yau theorem}, Sci. China Ser. A, {\bf 48} (2005), 244-247.

\bibitem{Br} T. Bridgeland, {\em Stability conditions on triangulated categories}, Ann. of. Math., {\bf 166} (2007), 317-345.

\bibitem{CC} X. Cabr\'e, and L. Caffarelli, {\em Fully nonlinear elliptic equations}, American Mathematical Society: Colloquium Publications, vol. 43. American Mathematical Socierty, Providence, R.I. (1995).

\bibitem{CNS} L. Caffarelli, L. Nirenberg, and J. Spruck, {\em The Dirichlet problem for nonlinear second order elliptic equations, III: Functions of the eigenvalues of the Hessian}, Acta. Math., {\bf 155} (1985), no. 3-4, 261-301.

\bibitem{Cao} H.-D. Cao, {\em Deformation of K\"ahler metrics to K\"ahler-Einstein metrics on compact K\"ahler manifolds}, Invent. Math. {\bf 81} (1985), 359-372.

\bibitem{Ch1} X.-X. Chen, {\em In the lower bound of the Mabuchi energy and its application}, Int. Math. Res. Notices, {\bf 12} (2000), 607-623.

\bibitem{Ch2} X.-X. Chen, {\em A new parabolic flow in K\"ahler manifolds}, Comm. Anal. Geom., {\bf 12} (2004), 837-852.

\bibitem{CW} K.-S. Chou, and X.-J. Wang, {\em A variational theory of the Hessian equation}, Comm. Pure Appl. Math., {\bf 54} (2001), 1029-1064.

\bibitem{Co} T. Collins, {\em $C^{2,\alpha}$ estimates for nonlinear elliptic equations of twisted type}, arXiv:1501.06455

\bibitem{CS} T. Collins, and G. Sz\'ekelyhidi, {\em Convergence of the $J$-flow on toric manifolds}, arXiv:1412:4809.

\bibitem{CT} T. Collins, and V. Tosatti, {\em K\"ahler currents and null loci}, Invent. Math., to appear.

\bibitem{DemB} J.-P. Demailly, {\em Complex Analytic and Differential Geometry}, available on the author's webpage.

\bibitem{DP} J.-P. Demailly, and M. P\u aun, {\em Numerical characterization of the K\"ahler cone of a compact K\"ahler manifold}, Ann. of Math., {\bf 159} (2004), no. 3, 1247-1274.

\bibitem{DK} S. Dinew, and S. Ko\l odziej, {\em Liouville and Calabi-Yau type theorems for complex Hessian equations}, arXiv:1203.3995.

\bibitem{Don1} S.K. Donaldson, {\em Anti self-dual Yang-Mills connections over complex algebraic surfaces and stable vector bundles}, Proc. London Math. Soc, {\bf 50} (1985), no. 3, 1-26

\bibitem{Don} S.K. Donaldson, {\em Moment maps and diffeomorphisms}, Asian J. Math., {\bf 3} (1999), 1-16.

\bibitem{Doug} M.R. Douglas, {\em $D$-branes on Calabi-Yau manifolds}, European Congress of Mathematics, Vol. II (Barcelona, 2000), 449-466, Progr. Math. {\bf 202}, Birkh\"auser, Basel, 2001.

\bibitem{Doug1} M.R. Douglas, {\em $D$-branes, categories and $N=1$ supersymmetry. Strings, branes, and M-theory}, J. Math. Phys.,  {\bf 42} (2001), 2818-2843.

\bibitem{Doug2} M.R. Douglas, {\em Dirichlet branes, homological mirror symmetry, and stability}, Proc. Internat. Congress of Mathematicians, Vol. III (Beijing, 2002), 395-408, Higher Ed. Press, Beijing, 2002.

\bibitem{E} L.C. Evans, {\em Classical solutions of fully nonlinear, convex, second order elliptic equations}, Comm. Pure. Appl. Math., {\bf 25} (1982), 333-363.

\bibitem{FLM} H. Fang, M. Lai, and X. Ma, {\em On a class of fully nonlinear flows in K\"ahler geometry}, J. Reine Angew. Math. {\bf 653} (2011), 189-220.

\bibitem{Guan} B. Guan, {\em Second order estimates and regularity for fully nonlinear elliptic equations on Riemannian manifolds}, Duke Math. J., {\bf 163} (2014), 1491-1524.

\bibitem{GL} B. Guan, and Q. Li, {\em The Dirichlet problem for a complex Monge-Amp\`ere type equation on Hermitian manifolds}, Adv. Math., {\bf 246} (2013), 351-367.

\bibitem{HL} R. Harvey, and H.B. Lawson, {\em Calibrated geometries}, Acta. Math., {\bf 148} (1982), 47-157.

\bibitem{HMW} Z. Hou, X.-N. Ma, and D. Wu, {\em A second order estimate for complex Hessian equations on a compact K\"ahler manifold}, Math. Res. Lett., {\bf 17} (2010), no. 3, 547-561.

\bibitem{JY} A. Jacob, and S.-T. Yau, {\em A special Lagrangian type equation for holomorphic line bundles}, arXiv:1411.7457.

\bibitem{Ko} S. Ko\l odziej, {\em The complex Monge-Amp\`ere equation and pluripotential theory}, Memoirs of the American Mathematical Society, {\bf 178}(2005), no. 840, Amer. Math. Soc. Providence, R.I. 

\bibitem{K} N.V. Krylov, {\em Boundedly nonhomogeneous elliptic and parabolic equations}, Izvestia Akad. Nauk. SSSR, {\bf 46} (1982), 487-523; English translation in Math. USSR Izv. {\bf 20} (1983), no. 3, 452-492.

\bibitem{LSU} O.A. Ladyzenska, V.A. Solonnikov, and N.N. Ural'Ceva, {\em Linear and quasilinear equations of parabolic type}, Translations of Mathematical Monographs, {\bf 23} (1968), Amer. Math. Soc., Providence, R.I.

\bibitem{LS} M. Lejmi, and G. Sz\'ekelyhidi, {\em The $J$-flow and stability}, Advances in Math., to appear.

\bibitem{LYZ} C. Leung, S.-T. Yau, and E. Zaslow, {\em From special Lagrangian to Hermitian-Yang-Mills via Fourier-Mukai transform}, Winter School on Mirror Symmetry, Vector Bundles and Lagrangian Submanifolds (Cambridge, MA, 1999), 209-225, AMS.IP Stud. Adv. Math., {\bf 23}, Amer, Math. Soc., Providence, RI, 2001.

\bibitem{LY} P. Li, and S.-T. Yau, {\em On the parabolic kernel of the Schr\"odinger operator}, Acta Math., {\bf 156}(1986), no.3-4, 153-201.

\bibitem{NV} N. Nadirashvili, and S. Vl\u adu\c{t}, {\em Singular solution to Special Lagrangian Equations}, Ann. Inst. H. Poincar\'e Anal. Non Lin\'eaire, {\bf 27} (2010), no. 5, 1179-1188.

\bibitem{P} V. Pingali, {\em A priori estimates for a generalized Monge-Amp\`ere PDE on some compact K\"ahler manifolds}, arXiv:1505.04358

\bibitem{Ron} L.I. Ronkin, {\em Introduction to the theory of entire functions of several variables}, Translations of Mathematical Monographs, {\bf 44}, Amer. Math. Soc., Providence, R.I., 1974.

\bibitem{RS} Y.A. Rubinstein, and J.P Solomon, {\em The degenerate special Lagrangian equation}, arXiv:1506:08077

\bibitem{Sm} K. Smoczyk, {\em Longtime existence of the Lagrangian mean curvature flow}, Calc. Var. Partial Differential Equations, {\bf 20} (2004), no. 1, 25-46.

\bibitem{SmW} K. Smoczyk, and M.-T. Wang, {\em Mean curvature flows of Lagrangian submanifolds with convex potentials}, J. Differential Geom., {\bf 62} (2002), no 2, 243-257.

\bibitem{Sol1} J.P. Solomon, {\em The Calabi homomorphism, Lagrangian paths and special Lagrangians}, Math. Ann., {\bf 357} (2013), 1389-1424.

\bibitem{Sol2} J.P Solomon, {\em Curvature of the space of positive Lagrangians}, geom. Funct. Anal., {\bf 24} (2014), 670-689. 

\bibitem{SW} J. Song, and B. Weinkove, {\em On the convergence and singularities of the $J$-flow with applications to the Mabuchi energy}, Comm. Pure Appl. Math., {\bf 61}(2008), 210-229.

\bibitem{WS} W. Sun, {\em On a class of fully nonlinear elliptic equations on closed Hermitian manifolds}, arXiv: 1310.0362.

\bibitem{Szek} G. Sz\'ekelyhidi, {\em Fully non-linear elliptic equations on compact hermitian manifolds}, arXiv:1501.02762v3.

\bibitem{SzTW} G. Sz\'ekelyhidi, V. Tosatti, and B. Weinkove, {\em Gauduchon metrics with prescribed volume form}, arXiv:1503.04991

\bibitem{T1} R.P. Thomas, {\em Moment maps, monodromy and mirror manifolds}, Symplectic geometry and mirror symmetry (Seoul, 2000), 467-498, World Sci. Publ., River Edge, NJ, 2001.

\bibitem{T} R.P. Thomas, {\em Stability conditions and the braid group}, Comm. Anal. Geom. {\bf 14} (2006), no. 1, 135--161.

\bibitem{TY} R.P. Thomas, and S.-T. Yau, {\em Special Lagrangians, stable bundles and mean curvature flow}, Comm. Anal. Geom., {\bf 10} (2002), no. 5, 1075-1113.

\bibitem{TWWY} V. Tosatti, Y. Wang, B. Weinkove, and X. Yang, {\em $C^{2,a}$ estimates for non-linear elliptic equations in complex and almost complex geometry}, Calc. Var. Partial Differential Equations, {\em to appear}.

\bibitem{UY} K. Uhlenbeck, and S.-T. Yau, {\em On the existence of Hermitian-Yang-Mills connections in stable vector bundles}, Comm. Pure Appl. Math., {\bf 39-S} (1986), 257-293.

\bibitem{WY} D. Wang, and Y. Yuan, {\em Hessian estimates for special Lagrangian equations with critical and supercritical phases in general dimensions}, Amer. J. Math., {\bf 136} (2014), 481-499.

\bibitem{WY1}  D. Wang, and Y. Yuan, {\em Singular solutions to the special Lagrangian equations with subcritical phases and minimal surface systems}, Amer. J. Math., {\bf 135} (2013), no. 5, 1157-1177.

\bibitem{MTW} M.-T. Wang, {\em Mean curvature flows and isotopy problems}, Surv. Differ. Geom., {\bf 18}, Int. Press, Somerville, MA, 2013.

\bibitem{YW} Y. Wang, {\em On the $C^{2,\alpha}$ regularity of the complex Monge-Amp\`ere equation}, Math. Res. Lett., {\bf 19} (2012), no. 4, 939-946.

\bibitem{W1} B. Weinkove, {\em Convergence of the $J$-flow on K\"ahler surfaces}, Comm. Anal. Geom., {\bf 12}(2004), 151-164.

\bibitem{W2} B. Weinkove, {\em On the $J$-flow in higher dimensions and the lower boundedness of the Mabuchi energy}, J. Differential Geom., {\bf 73} (2006), 351-358.

\bibitem{Yau} S.-T. Yau, {\em On the Ricci curvature of compact K\"ahler manifolds and the complex Monge-Amp\`ere equation, I}, Comm. Pure Appl. Math., {\bf 31} (1978), no. 3, 339-411.

\bibitem{Y1} Y. Yuan, {\em A Bernstein problem for special Lagrangian equations}, Invent. Math., {\bf 150} (2002), no. 1, 117-125.

\bibitem{Y} Y. Yuan, {\em Global solutions to special Lagrangian equations}, Proc. Amer. Math. Soc., {\bf 134} (2006), no. 5, 1355-1358.


\end{thebibliography}
\end{document}